\documentclass[12pt,a4paper]{article}
\usepackage{latexsym,amssymb,amsmath,a4wide,theorem}
\usepackage[a4paper, total={6in, 10in}]{geometry}
\usepackage[isolatin]{inputenc}
\usepackage{mathrsfs}

\newcommand{\R}{\mathbb{R}}

\usepackage[T1]{fontenc}
\usepackage{times}
\usepackage{cite}
\usepackage{amsfonts}
\usepackage{graphicx}
\usepackage[all]{xy}
\usepackage{esint}
\usepackage[titletoc,title]{appendix}
\usepackage{indentfirst}
\usepackage{float}
\usepackage{cancel}
\usepackage{booktabs}
\usepackage{makecell}
\usepackage{color}
\usepackage{epstopdf}
\usepackage{hyperref}
\def\sqr#1#2{{\vbox{\hrule height.#2pt
     \hbox{\vrule width.#2pt height#1pt \kern#1pt
           \vrule width.#2pt}
     \hrule height.#2pt}}}
\def\square{\sqr74}
\def\qed{{\unskip\nobreak\hfil\penalty50\hskip1em
             \hbox{}\nobreak\hfil\square \parfillskip=0pt
             \finalhyphendemerits=0 \par\goodbreak \vskip8mm}}

\newenvironment{proof}{\medskip\par\noindent{\bf Proof\/}.\quad}{\qquad
\raisebox{-0.5mm}{\rule{1.5mm}{1mm}}\vspace{6pt}}

\newtheorem{Thm}{Theorem}[section]

\newtheorem{Lem}{Lemma}[section]

\newtheorem{Rek}{Remark}[section]

\numberwithin{equation}{section}

\title{{Multiplicity and concentration of dual solutions for a Helmholtz system}\thanks{This paper was supported by NSFC(12261107) and Yunnan Key Laboratory of Modern Analytical Mathematics and Applications (No. 202302AN360007).}}

\author{Ruowen Qiu$^{1,2}$,~ Fei Yuan$^{2}$~and~Fukun Zhao$^{2,3}$\thanks{Corresponding author. E-mall address: fukunzhao@163.com}\\
1.Department of Mathematics, Zhejiang Normal University,\\
2.Department of Mathematics, Yunnan Normal University,\\
3.Yunnan Key Laboratory of Modern Analytical Mathematics and Applications\\
 Cheng Gong, Kunming, 650500}
\date{}

\date{}
\begin{document}
\maketitle
\noindent{\bf Abstract}
\quad In this paper, we are concerned with the nonlinear Helmholtz system of Hamiltonian type
\begin{equation*}
\left\{\begin{array}{l}
-\Delta u-k^2 u=P(x)|v|^{p-2}v,\quad \text{in}\ \R^N, \\
-\Delta v-k^2v=Q(x)|u|^{q-2}u,\quad \text{in}\ \R^N,
\end{array} \right.
\end{equation*}
where $N\geq3$, $P,Q:\R^N\rightarrow\R$ are two positive continuous functions, the  exponents $p,q>2$ satisfy $\frac{1}{p}+\frac{1}{q}>\frac{N-2}{N}$. First, we obtained the existence of a ground state solution via a dual variational method.
Moreover, the concentration behavior of such dual ground state solutions is established
 as $k\rightarrow\infty$, where a rescaling technique and the generalized Birman-Schwinger operator are involved. In addition,  we also investigated the relation between the number of solutions and the topology of the set of the global maxima of the functions $P$ and $Q$.\\

\numberwithin{equation}{section}

\section{Introduction and main result}
The Helmholtz equation is a fundamental partial differential equation used in physics to describe various phenomena in fields such as acoustics, electromagnetism, and quantum mechanics. It can be understood as a special case of the scalar equation, such as:
\begin{equation}\label{eq1.0}
-\Delta u-k^2u=Q(x)|u|^{p-2}u.
\end{equation}
The meaning of the proportionality constant $k$ depends on the physical problem addressed.
One reason for the great importance of the Helmholtz equation is that in dynamic analysis, it can be interpreted as the time harmonic representation of the standing wave (or solitary wave) solutions  $\psi(t, x) = e^{ikt} u(x)$ of the nonlinear Klein-Gordon equation given by:
\begin{equation*}
	\frac{\partial^2 \psi}{\partial t^2}(t, x) - \Delta \psi(t, x) = Q(x)|\psi(t, x)|^{p-2}\psi(t, x), \quad (t, x) \in \mathbb{R} \times \mathbb{R}^N.
\end{equation*}
\par
The variational method is a powerful tool to study the existence of solutions of scalar equations. However, it seems that a direct variational functional corresponding to 
the equation \eqref{eq1.0} does not exists since $0$ is contained in the essential spectrum of the Schr\"{o}dinger operator $-\Delta-k^2$. So, how to apply the variational method to obtain the existence of (weak) solutions of \eqref{eq1.0} is a challenging issue. To overcome the difficulty that the Schr\"{o}dinger operator $-\Delta-k^2$ is not invertible, in their pioneering paper \cite{Evequoz-Weth2014ARMA}, Ev\'equoz and Weth considered the existence and multiplicity of real solutions of \eqref{eq1.0} with compactly supported nonlinearities, and transferred the original equation into an elliptic equation settled on a bounded domain. Therefore, a solution, and if the inhomogeneous source term of the equation is odd with respect to the second variable, a sequence of pairs of solutions of the equation \eqref{eq1.0} was obtained via the (symmetrical) linking theorem, and the norms of which go to infinity. To remove the restriction that the nonlinearity has compact support, in another interesting paper \cite{Evequoz-Weth2015AM}, Ev\'equoz and Weth developed a dual variational framework for \eqref{eq1.0}. More precisely, by using the properties of the fundamental solution of 
\begin{equation*}
  -\Delta \Phi-\Phi=\delta,
\end{equation*}
some estimates were established for the resolvent operator $\mathscr{R}$ (see subsection 2.1), and the original equation \eqref{eq1.0} was transferred into the dual integral equation
\begin{equation*}
  u=\mathbf{R}(Q(x)|u|^{p-2}u),~u\in L^p(\mathbb{R}^N),
\end{equation*}
where $\mathbf{R}$ denotes the real part of the resolvent operator $\mathscr{R}$. Hence, the existence of nontrivial solutions of \eqref{eq1.0} with periodic $Q$ as well as in the vanishing case where $Q(x)\to 0$ as $|x|\to\infty$, where $N\geq3$, $\frac{2(N+1)}{(N-1)}<p<\frac{2 N}{N-2}$ and $Q\in L^{\infty}\left(\mathbb{R}^N\right)$ is nonnegative. A new nonvanishing theorem related to the dual integral equation plays a key role in the proof. 
\par
After then, the existence and multiplicity of solutions to the nonlinear Helmholtz equation \eqref{eq1.0} have been extensively studied, particularly when $Q(x)$ is periodic or asymptotically periodic\cite{Evequoz2017NA}, $Q$ satisfies the global maximum condition\cite{Evequoz2017AMPA}, $Q(x)$ is sign-changing\cite{Mandel-Scheider-Yecsil2021CVPDE}, and in the critical case \cite{Evequoz-Yecsil2020}. In particular, in \cite{Evequoz2017AMPA}, 
 for large frequency number $k>0$, Ev\'equoz obtained the existence and multiplicity of solutions of \eqref{eq1.0}, where the exponent $p$ is subcritical, and the weight function $Q$ is continuous, nonnegative and satisfies the condition
\begin{equation*}
  \limsup\limits_{|x|\to\infty}Q(x)<\sup\limits_{x\in\mathbb{R}^N}Q(x).
\end{equation*}
Moreover, in the limit $k\to\infty$, sequences of solutions associated with ground states of the dual integral equation are shown to concentrate, after rescaling, at global maximum points of the function $Q$. It seems that this paper is the only known work concerning the asymptotical behavior of dual solutions of \eqref{eq1.0}. Other results for the Helmholtz equation \eqref{eq1.0} can be found in  \cite{Bonheure-Casteras-Mandel2019JMLS,Casteras-Mandel2021IMRN,Weth-Yecsil2021ANPA,Chen-Evequoz-Weth2021SIAM}.
\par
Similar to the Schr\"{o}dinger system, which appears in nonlinear optics and condensed matter physics, the Helmholtz system is particularly interesting from a viewpoint of mathematics. However, as far as we know, there are only a few works concerned with this topic, except for \cite{Mandel-Scheider2018NoDEA, Mandel-Scheider2020ANA, Ding-Wang2023JDE}.
In \cite{Mandel-Scheider2018NoDEA, Mandel-Scheider2020ANA}, Mandel and Scheider considered a nonlinear Helmholtz system of gradient type which is similar to the Schr\"{o}dinger system considered by Lin and Wei \cite{Lin-Wei2005} and \cite{Lin-Wei2006JDE}. In \cite{Ding-Wang2023JDE}, Ding and Wang considered a nonlinear and subcritical Helmholtz system of Hamiltonian type, which is related to the classical Lane-Emden system (see e.g. \cite{Li-Wei-Zou2023JMPA}). Here we refer to \cite{deFigueiredo1994TAMS} and \cite{Hulshof-Robertus1993JFA} for the critical hyperbola for elliptic systems of Hamiltonian type. A dual variational framework was established via a generalized Birman-Schwinger operator $\mathbf{K}_{p, q}^{P, Q}$, and the existence and multiplicity of dual solutions were obtained both for the periodic potential case and the compact potential case. A nonvanishing theorem was established to recover the loss of compactness. It should be noted that the boundedness of the linear operator $\mathbf{K}_{p, q}^{P, Q}$ defined here relies on an $L^\alpha-L^\beta$ estimate of Guti\'errez for the resolvent operator $\mathscr{R}$, see \cite[Theorem 6]{Gutierrez2004MA}. It seems that this is the only known result for the Helmholz system of Hamiltonian type.
\par
Motivated by \cite{Ding-Wang2023JDE} and \cite{Evequoz2017AMPA}, in this paper, we study the existence, multiplicity and the concentration behavior of dual solutions of the nonlinear Helmholtz system 
\begin{equation}\label{eq1.1}
\left\{\begin{array}{l}
-\Delta u-k^2 u=P(x)|v|^{p-2}v,\quad \text{in}\  \R^N, \\
-\Delta v-k^2v=Q(x)|u|^{q-2}u,\quad\text{in}\ \R^N,
\end{array} \right.
\end{equation}
where $k>0$, $N\geq3$, $P,Q\in L^\infty(\R^N)$ are two positive continuous functions, the  exponents $p,q$ satisfy 
$$\frac{2 N}{N-1}<p,q \leq \infty,\quad  \frac{N-2}{N}<\frac{1}{q}+\frac{1}{p}<\frac{N-1}{N+1}.$$
\par
There are some difficulties in our system \eqref{eq1.1}.
First, \eqref{eq1.1} does not have a direct variational framework. To see this,
we set $k=1$, so the system \eqref{eq1.1} can be expressed in the form:
$$\mathcal{A}\binom{v}{u}:=\left(\begin{array}{cc}
0 & -\Delta-1 \\
-\Delta-1 & 0
\end{array}\right)\binom{v}{u}=\binom{P(x)|v|^{p-2} v}{Q(x)|u|^{q-2} u} .$$
The primary challenge arises from the fact that $\sigma(\mathcal{A})=\sigma_{\mathrm{ess}}(\mathcal{A})=\mathbb{R}$, where $\sigma(\mathcal{A})$ and $\sigma_{\text{ess}}(\mathcal{A})$ denote the spectrum and the essential spectrum of the operator $\mathcal{A}$, respectively. Consequently, the operator $-\Delta-1$ is not invertible, which means one can not find out the variational functional corresponding to the system \eqref{eq1.1}. Second, note that the exponents $p$ and $q$ can be different, which results in two different dual relations when rescaling. Therefore, some previous results on the Helmholtz equation \eqref{eq1.0}, such as the estimation of the nonlocal term appearing in the dual energy functional\cite[Lemma 2.4]{Evequoz2017AMPA}, are no longer applicable to our problem. Third, since the weight functions $P(x)$ and $Q(x)$ are allowed to be different, a competition between $P(x)$ and $Q(x)$ may occur when studying the concentration of dual solutions. So it is not easy to determine the concentration set of dual solutions, which is related to the properties of $P(x)$ and $Q(x)$. Unlike the known dual variational methods, such as \cite{Alves-Soares-Yang2003ANS}, we need to know more properties of the operator $\mathbf{R}$ and make more accurate estimates.
\par
To state our main results, some notations are in order.
$$M_P=\{x\in \R^N:P(x)=\bar{P}\},\quad   M_Q=\{x\in \R^N:Q(x)=\bar{Q}\}, $$
where $\bar{P}:=\sup\limits_{x\in\R^N}P(x)$ and $\bar{Q}:=\sup\limits_{x\in\R^N}Q(x)$. We assume
\begin{itemize}
\item[$(PQ_1)$]  $P,Q\in L^\infty(\R^N)$ are continuous functions and $P(x),Q(x)\geq\alpha>0$ for all $x\in\R^N$;
\item[$(PQ_2)$]  $\bar{P}>P_\infty:=\limsup\limits_{|x|\rightarrow\infty}P(x)$ and $\bar{Q}>Q_\infty:=\limsup\limits_{|x|\rightarrow\infty}Q(x)$;
\item[$(PQ_3)$] $M_P\cap M_Q\neq\emptyset$.
\end{itemize}
\par
Note that $(PQ_1)$ and $(PQ_2)$ imply that the sets $M_P$ and $M_Q$ are bounded, and hence $M=\{(x,x):x\in M_P\cap M_Q \}\neq\emptyset$ is bounded.
For any $\delta>0$, let us denote by $M_\delta=\{(x,x)\in\R^N\times\R^N:\text{dist}((x,x),M)\leq\delta\}$ the $\delta$-neighborhood of $M$ in $\R^N\times\R^N$. Additionally, if $Y$ is a closed subset of a topological space $X$, then $\text{cat}_X(Y)$ denotes the Ljusternik-Schnirelman category of $Y$ in $X$, i.e., the least number of closed and contractible sets in $X$ which cover $Y$. 
Now we state our main results as follows.
\begin{Thm}\label{Thm1.1}
Let $N \geq 3$ and $p$, $q>\frac{2 N}{N-1}$ with $\frac{N-2}{N}<\frac{1}{p}+\frac{1}{q}<\frac{N-1}{N+1}$. If $(PQ_1)$$-$$(PQ_3)$ hold, then there is $k_0>0$ such that for all $k>k_0$ the system \eqref{eq1.1} admits a pair of dual ground state. In addition, let $\{k_n\}\subset(k_0,\infty)$ with $k_n\rightarrow\infty$ as $n\rightarrow\infty$, and consider for each $n$, a pair of dual ground state $(u_n,v_n)$ of 
\begin{equation*}
\left\{\begin{array}{l}
-\Delta u-k_n^2 u=P(x)|v|^{p-2}v,\quad \text{in}\  \R^N, \\
-\Delta v-k_n^2v=Q(x)|u|^{q-2}u,\quad\text{in}\ \R^N.
\end{array} \right.
\end{equation*}
Then there exists a pair of $(x_0,x_0)\in M$, a dual ground state $(u_0,v_0)$ of 
\begin{equation}\label{eq1.5}
\left\{\begin{array}{l}
-\Delta u- u=\bar{P}|v|^{p-2}v,\quad \text{in}\  \R^N, \\
-\Delta v-v=\bar{Q}|u|^{q-2}u,\quad \text{in}\  \R^N,
\end{array} \right.
\end{equation}
and a sequence $\{x_n\}\subset\R^N$ such that, up to a subsequence, $\lim\limits_{n\rightarrow\infty}x_n=x_0\in M_P\cap M_Q $ and
$$k_n^{\frac{2p}{(q-1)(p-1)-1}}u_n\bigg(\frac{\cdot}{k_n}+x_n\bigg)\rightarrow u_0\ \text{in}\ L^{q}(\R^N),\quad k_n^{\frac{2q}{(q-1)(p-1)-1}}v_n\bigg(\frac{\cdot}{k_n}+x_n\bigg)\rightarrow v_0\ \text{in}\ L^{p}(\R^N)$$
as $n\rightarrow\infty$.
\end{Thm}
\begin{Rek}
For the definition of the dual ground state of system \eqref{eq1.1} in Theorem \ref{Thm1.1}, we refer readers to see the definition in \cite{Evequoz2017AMPA}. Specifically, we seek to find the nontrivial critical point $(\psi,\phi)$ of the energy functional $J_\varepsilon$ (see subsection 2.1) associated with the corresponding integral system \eqref{eq2.4} at the mountain pass level. 
\end{Rek}

\begin{Thm}\label{Thm1.2} Under the assumptions of Theorem \ref{Thm1.1}, for every $\delta>0$, there exists $k(\delta)>0$ such that the system \eqref{eq1.1} has at least $\textup{cat}_{M_\delta}(M)$ pairs of nontrivial solutions for all $k>k(\delta)$.
\end{Thm}
\par
The paper is organized as follows. In Section 2, the dual variational framework of the Helmholtz system \eqref{eq1.1} is established. Additionally, some preliminary results are also studied in this section. Section 3 is devoted to studying the Palais-Smale condition for the dual energy functional on the Nehari manifold below some limit energy level, which is based on a crucial estimate of the quadratic part of the functional, see Lemma \ref{Lem3.1}. To obtain the concentration of the dual ground state, we also establish a representation Lemma for the Palais-Smale sequences of the energy functional $J_{s,t}$ corresponding to the constant coefficient system. The proofs of Theorem \ref{Thm1.1} and \ref{Thm1.2} are given in Section 4.
\par
Throughout this paper, let $B_r(x)$ be the open ball in $\R^N$ of radius $r$ centered at $x$ and $B_r=B_r(0)$. For $1\leq q\leq\infty$, we write $\|\cdot\|_q$ instead of $\|\cdot\|_{L^q(\R^N)}$ for the standard norm of the Lebesgue space $L^q(\R^N)$. The symbols $\mathscr{S}$ and $\mathscr{S}^\prime$ denote the Schwartz space and the space of tempered distributions on $\mathbb{R}^N$, respectively. For $f\in\mathscr{S}^\prime$, we write $\widehat{f}$ to denote the Fourier transform of $f$. For any measurable set $E\subset \R^N$, $1_E$ denotes the characteristic function of the set $E$. In addition, we shall denote the conjugate exponent of $p$ by $p^\prime$ satisfying $\frac{1}{p}+\frac{1}{p^\prime}=1$. Finally, $C$ and $C_i(i=1,2,3,\cdots)$ are some positive constants that may change from line to line.
\section{The variational formulation and preliminary results}
\subsection{ Dual variational formulation}
Setting $\varepsilon= k^{-1}$, $u=\varepsilon^{\beta_1}\hat{u}$ and $ v=\varepsilon^{\beta_2}\hat{v}$, \eqref{eq1.1} can be rewritten as
 \begin{equation}\label{eq2.1}
\left\{\begin{array}{l}
-\varepsilon^2\Delta \hat{u}-\hat{u}=P(x)|\hat{v}|^{p-2}\hat{v},\quad \text{in}\ \R^N, \\
-\varepsilon^2\Delta \hat{v}-\hat{v}=Q(x)|\hat{u}|^{q-2}\hat{u},\quad\text{in}\ \R^N,
\end{array} \right.
\end{equation}
where $\beta_1=\frac{2p}{1-(q-1)(p-1)}$, $\beta_2=\frac{2q}{1-(q-1)(p-1)}$. Note that $\beta_1$ and $\beta_2$ are well-defined because of the range of $p,q$.
In addition,  if $(u,v)$ is a solution of the system
\begin{equation}\label{eq2.2}
\left\{\begin{array}{l}
-\Delta u- u=P_\varepsilon(x)|v|^{p-2}v,\quad \text{in}\ \R^N, \\
-\Delta v-v=Q_\varepsilon(x)|u|^{q-2}u,\quad \text{in}\ \R^N,
\end{array} \right.
\end{equation}
where $P_\varepsilon(x)=P(\varepsilon x)$ and $Q_\varepsilon(x)=Q(\varepsilon x)$. Then $(u(\varepsilon^{-1}x),v(\varepsilon^{-1}x))$ is a solution of the system \eqref{eq2.1}. Thus, in the sequel, we will study the equivalent system \eqref{eq2.2}.

To establish the dual variational framework of \eqref{eq2.2}, let us review some results from \cite{Ding-Wang2023JDE}. Let $\varepsilon>0$, the operator $-\Delta-(1+i\varepsilon):H^2(\R^N)\subset L^2(\R^N)\rightarrow L^2(\R^N)$ is an isomorphism. For any $f\in\mathscr{S}$, we can define the inverse of the above operator as follows 
$$\mathscr{R}_\varepsilon f(x)=[-\Delta-(1+i \varepsilon)]^{-1} f(x)=(2 \pi)^{-\frac{N}{2}} \int_{\mathbb{R}^N} e^{i x \cdot \xi} \frac{\hat{f}(\xi)}{|\xi|^2-(1+i \varepsilon)} \mathrm{d} \xi .$$
Moreover, there exists a linear operator $\mathscr{R}: \mathscr{S} \rightarrow \mathscr{S}^{\prime}$ given by
$$
\langle\mathscr{R} f, g\rangle:=\lim _{\varepsilon \rightarrow 0} \int_{\mathbb{R}^N}\left[\mathscr{R}_{\varepsilon} f\right](x) g(x) \mathrm{d} x=\int_{\mathbb{R}^N}[\Phi * f](x) g(x) \mathrm{d} x \quad \text { for } f, g \in \mathscr{S}
$$
with
$$
\Phi(x):=\frac{i}{4}(2 \pi|x|)^{\frac{2-N}{2}} H_{\frac{N-2}{2}}^{(1)}(|x|), \quad\forall x\in\R^N\setminus\{0\},
$$
where $\Phi(x)$ is also the fundamental solution of the Helmholtz equation $-\Delta u-u=\delta$ and $H_k^{(1)}$ is the Hankel function of the first kind of order $k$.  For any $f\in \mathscr{S}$, the function $u=\mathscr{R} f =\Phi\ast f\in \mathscr{C}^{\infty}\left(\mathbb{R}^N\right)$ is a solution of the inhomogeneous Helmholtz equation $-\Delta u-u=f$ satisfying some decay conditions. Our aim is to investigate the existence of real-valued solutions of \eqref{eq2.1} for $k>0$ large enough. Since $K(x)$ and $Q(x)$ are real numbers, we can infer that 
$$\Re(u)=\Re(\mathscr{R} f)=\Re((\Phi\ast f))=\Re(\Phi)\ast f.$$ 

Putting the real part of $\Phi$ by $\Psi$, as a consequence of \cite[Theorem 6]{Gutierrez2004MA}, we know the integral operator
$$\mathbf{R}:L^\alpha(\R^N)\rightarrow L^\beta(\R^N),\ \mathbf{R}f=\Psi\ast f$$ 
is bounded if $\frac{2}{N+1}\leq \frac{1}{\alpha}-\frac{1}{\beta}\leq\frac{2}{N}$, $\frac{1}{\alpha}>\frac{N+1}{2N}$ and $\frac{1}{\beta}<\frac{N-1}{2N}$, see also \cite[Theorem 2.1]{Ding-Wang2023JDE}.
\par
For $(u,v)\in L^{q}(\mathbb{R}^N)\times L^{p}(\mathbb{R}^N)$, it is easy to see that 
\begin{equation}\label{eq2.3}
\psi=Q_\varepsilon^{\frac{1}{q^\prime}}|u|^{q-2}u\in L^{q^\prime}(\R^N), \quad\phi=P_\varepsilon^{\frac{1}{p^\prime}}|v|^{p-2}v\in L^{p^\prime}(\R^N), 
\end{equation}  
and we will consider the following integral system 
\begin{equation}\label{eq2.4}
\left\{\begin{array}{l}
|\psi|^{q^\prime-2}\psi=Q_\varepsilon^{\frac{1}{q}}\mathbf{R}(P_\varepsilon^\frac{1}{p}\phi), \\
|\phi|^{p^\prime-2}\phi=P_\varepsilon^{\frac{1}{p}}\mathbf{R}(Q_\varepsilon^\frac{1}{q}\psi).
\end{array} \right.
\end{equation}
Note that the functions $P(x)$ and $Q(x)$ are different, we need the following definition of the generalized Birman-Schwinger Operator and some of its properties \cite[Lemma 2.3]{Ding-Wang2023JDE}:
\begin{Lem}\label{Lem2.1}
 Let $N \geq 3, p, q>\frac{2 N}{N-1}$ with $\frac{N-2}{N} \leq\frac{1}{p}+\frac{1}{q} \leq \frac{N-1}{N+1}$ and consider nonnegative functions $P, Q \in L^{\infty}\left(\mathbb{R}^N\right) \backslash\{0\}$, the generalized Birman-Schwinger operator is defined in the following way:

$$
\mathbf{K}_{p, q}^{P,Q}: L^{p^{\prime}}\left(\mathbb{R}^N\right) \rightarrow L^q\left(\mathbb{R}^N\right), \quad \mathbf{K}_{p, q}^{P,Q}(v)=P^{\frac{1}{p}} \mathbf{R}\left(Q^{\frac{1}{q}} v\right) .
$$
In addition, the following propositions hold:
\begin{itemize}
\item[$(a)$]  $\int_{\mathbb{R}^N} u \mathbf{K}_{p, q}^{P,Q}(v) dx=\int_{\mathbb{R}^N} v \mathbf{K}_{q, p}^{Q,P}(u) dx \text { for all } v \in L^{p^{\prime}}\left(\mathbb{R}^N\right), u \in L^{q^{\prime}}\left(\mathbb{R}^N\right)$ ;
\item[$(b)$] If $\frac{N-2}{N} <\frac{1}{p}+\frac{1}{q} \leq \frac{N-1}{N+1}$, then for any bounded and measurable set $B\subset\R^N$, the operator  $1_B \mathbf{K}_{p, q}^{P,Q}$ is compact. 
\end{itemize}
\end{Lem}

For the sake of simplicity, we define a bounded linear operator $\mathbf{K}_\varepsilon: L^{q^{\prime}}\left(\mathbb{R}^N\right)\times L^{p^{\prime}}\left(\mathbb{R}^N\right)\rightarrow L^{q}\left(\mathbb{R}^N\right)\times L^{p}\left(\mathbb{R}^N\right)$ and \eqref{eq2.4} can be written as
$$
\mathbf{K}_\varepsilon z:=\left(\begin{array}{cc}
0 & \mathbf{K}_{q, p}^{Q_\varepsilon, P_\varepsilon} \\
\mathbf{K}_{p, q}^{P_\varepsilon, Q_\varepsilon} & 0
\end{array}\right)\binom{\psi}{\phi}=\binom{|\psi|^{q^\prime-2}\psi}{|\phi|^{p^\prime-2}\phi},
$$
where $z=(\psi,\phi)^T$. 
Set $X=L^{q^\prime}(\R^N)\times L^{p^\prime}(\R^N)$ with the norm
$$\|z\|=(\|\psi\|^2_{q^\prime}+\|\phi\|^2_{p^\prime})^{\frac{1}{2}},\ z=(\psi,\phi)^T\in X,$$
and it is easy to show that $(X,\|\cdot\|)$ is a reflexive Banach space. The dual space of $X$ is given by $X^\ast=L^{q}(\R^N)\times L^{p}(\R^N)$.
Define the $C^1$-functional $J_\varepsilon : X\rightarrow \R$ by
\begin{equation}\label{eq2.5}
J_\varepsilon(z)=\frac{1}{q^\prime}\int_{\R^N}|\psi|^{q^\prime}dx +\frac{1}{p^\prime}\int_{\R^N}|\phi|^{p^\prime}dx-\frac{1}{2}\int_{\R^N}z^T\mathbf{K}_\varepsilon zdx.\end{equation}
Note that if we find the nontrivial critical point $z=(\psi,\phi)^T$ of $J_\varepsilon$, then the following correspondences
\begin{equation}\label{eq2.6}
u=\mathbf{R}(P_\varepsilon^{\frac{1}{p}}\phi),\quad v=\mathbf{R}(Q_\varepsilon^{\frac{1}{q}}\phi)
\end{equation} implies that $(u,v)$ is indeed a solution of the following system
\begin{equation}\label{eq2.7}
\left\{\begin{array}{l}
u=\mathbf{R}(P_\varepsilon|v|^{p-2}v), \\
v=\mathbf{R}(Q_\varepsilon|u|^{q-2}u).
\end{array} \right.
\end{equation}
In addition, since $\mathbf{R}$ is a right inverse for the Helmholtz operator $-\Delta-1$ and the regularity result of \cite[Proposition A.1]{Ding-Wang2023JDE}, we know that $(u,v)\in W_{loc}^{2,p^\prime}(\R^N)\times W_{loc}^{2,q^\prime}(\R^N)$ is a pair of strong solution for the system \eqref{eq2.2}. 
\begin{Lem}\textup{(Mountain Pass geometry\cite[Lemma 2.5]{Ding-Wang2023JDE})}\label{Lem2.2}
Let $p, q>\frac{2 N}{N-1}$ with $\frac{N-2}{N}<\frac{1}{p}+\frac{1}{q} \leq \frac{N-1}{N+1}$.  If (PQ1) holds, then we have 
\begin{itemize}
\item[$(a)$]  there exists $\delta > 0$ and $0<\rho<1$ such that $J_\varepsilon(z) \geq \delta>0$ for all $z \in X$ with $\|z\|=\rho$;
\item[$(b)$]  there exists $z_0 \in X$ such that $\left\|z_0\right\|>1$ and $J_\varepsilon\left(z_0\right)<0$.
\end{itemize}
\end{Lem}
As a consequence of Lemma \ref{Lem2.2}, we know that the Nehari set 
$$\mathcal{N}_{\varepsilon}:=\{z \in X \setminus\{0\}: \langle J_{\varepsilon}^{\prime}(z), z \rangle
=0\}\neq\emptyset,$$
where the set $\mathcal{N}_{\varepsilon}$ contains all nontrivial critical points of $J_\varepsilon$. Here we would like to point out that $\mathcal{N}_{\varepsilon}$ is a smooth manifold. In fact, let $g(z)=\langle J_{\varepsilon}^{\prime}(z), z \rangle$, for $z=(\psi,\phi)^T\in \mathcal{N}_{\varepsilon}$. Then 
$$\langle g^{\prime}(z), z \rangle=q^\prime\int_{\mathbb{R}^N}|\psi|^{q^\prime } dx+p^\prime\int_{\mathbb{R}^N}|\phi|^{p^\prime} dx-2\int_{\mathbb{R}^N}z^T\mathbf{K}_\varepsilon z dx.$$
Using the characterization of $\mathcal{N}_{\varepsilon}$, we get
$$\langle g^{\prime}(z), z \rangle=(q^\prime-2)\int_{\mathbb{R}^N}|\psi|^{q^\prime } dx+(p^\prime-2)\int_{\mathbb{R}^N}|\phi|^{p^\prime} dx.$$
Because $z\neq0$ and $1<q^\prime,p^\prime<2$, we know that $\langle g^{\prime}(z), z \rangle<0$ for all $z\in\mathcal{N}_{\varepsilon}$, it follows from the implicit function theorem
that $\mathcal{N}_{\varepsilon}$ is a $C^1$-manifold of codimension 1 and $X=T_z(\mathcal{N}_\varepsilon)\oplus\mathbb{R}z$ for each $z\in\mathcal{N}_\varepsilon$, where $T_z(\mathcal{N}_\varepsilon)=\{w\in X: \langle g^\prime(z),w \rangle=0
\}$.
\par
In addition, for each $z\in U_\varepsilon^+:=\left\{z \in X: \int_{\mathbb{R}^N} z^T \mathbf{K}_\varepsilon z dx>0\right\}$ there exists a unique $t_z>0$ such that $t_z z\in\mathcal{N}_{\varepsilon}$. According to the existing results about the Nehari set $\mathcal{N}_{\varepsilon}$ obtained in \cite[Lemma 2.13 and 2.14]{Ding-Wang2023JDE}, it follows that 
$$c_\varepsilon:=\inf\limits_{z\in \mathcal{N}_{\varepsilon}}J_\varepsilon(z)=\inf\limits_{z\in U_\varepsilon^+}J_\varepsilon(t_zz)>0. $$
Moreover, $c_\varepsilon$ coincides with the mountain pass level, i.e., $c_\varepsilon=\inf\limits_{\gamma \in \Gamma} \max\limits_{t \in[0,1]} J_\varepsilon(\gamma(t))$, where $\Gamma=\{\gamma \in C([0,1], X): \gamma(0)=0$ and $J_\varepsilon(\gamma(1))<0\}$.
\par
We recall that $\{z_n\}=\{(\psi_n,\phi_n)^T\}\subset X$ is called a (PS)$_c$-sequence for $J_\varepsilon$ if $J_\varepsilon(z_n)\rightarrow c$ and $J^\prime_\varepsilon(z_n)\rightarrow 0$ as $n\rightarrow\infty$, and we say the functional $J_\varepsilon$ satisfies the (PS)$_c$-condition if any (PS)$_c$-sequence has a convergent subsequence. For the convenience of subsequent discussion, we now present some properties of the (PS)$_c$-sequence of $J_\varepsilon$ at $c>0$.
\begin{Lem}\label{Lem2.2+}
Let $p, q>\frac{2 N}{N-1} \text { with } \frac{N-2}{N}<\frac{1}{p}+\frac{1}{q} <\frac{N-1}{N+1}$ and $\{z_n\}=\{(\psi_n,\phi_n)^T\}\subset X$ be a (PS)$_c$-sequence for $J_\varepsilon$ at $c>0$, then $\{z_n\}$ is bounded in $X$ and there exists $z\in X$ such that, up to a subsequence, $z_n\rightharpoonup z$ weakly in $X$ and $J^\prime_\varepsilon(z)=0$. In addition, we have $J_\varepsilon(z)\leq\liminf\limits_{n\rightarrow\infty}J_\varepsilon(z_n)$.
\end{Lem}
\begin{proof}
The boundedness of $\{z_n\}$ can be derived from \cite[Lemma 2.6 (a)]{Ding-Wang2023JDE}. Hence, up to subsequence, we may assume that $z_n\rightharpoonup z$ weakly in $X$ for some $z:=(\psi,\phi)^T\in X$, by \cite[Lemma 2.6 (b)]{Ding-Wang2023JDE} we know that $J^\prime_\varepsilon(z)=0$. Combining this with the fact that the norm $\|\cdot\|_{p^\prime}$ and $\|\cdot\|_{q^\prime}$ are weakly lower sequentially continuous, we infer that
\begin{equation*}
\begin{split}
J_\varepsilon(z)&=J_\varepsilon(z)-\frac{1}{2}\langle J^\prime_\varepsilon(z),z \rangle\\
&=\bigg(\frac{1}{q^\prime}-\frac{1}{2}\bigg)\int_{\mathbb{R}^N}|\psi|^{q^\prime}dx+\bigg(\frac{1}{p^\prime}-\frac{1}{2}\bigg)\int_{\mathbb{R}^N}|\phi|^{p^\prime}dx\\
&\leq\liminf_{n\rightarrow\infty}\bigg[\bigg(\frac{1}{q^\prime}-\frac{1}{2}\bigg)\int_{\mathbb{R}^N}|\psi_n|^{q^\prime}dx+\bigg(\frac{1}{p^\prime}-\frac{1}{2}\bigg)\int_{\mathbb{R}^N}|\phi_n|^{p^\prime}dx\bigg]\\
&=\liminf_{n\rightarrow\infty}\bigg(J_\varepsilon(z_n)-\frac{1}{2}\langle J^\prime_\varepsilon(z_n),z_n \rangle\bigg)\\
&=\liminf_{n\rightarrow\infty}J_\varepsilon(z_n),
\end{split}
\end{equation*}
and this concludes the proof.
\end{proof}
\subsection{ Preliminaries}
At the beginning of this subsection, we first consider the following problem with constant coefficients
\begin{equation}\label{eq2.8}
\left\{\begin{array}{l}
|\psi|^{q^\prime-2}\psi=Q^{\frac{1}{q}}(s)\mathbf{R}(P^\frac{1}{p}(t)\phi), \\
|\phi|^{p^\prime-2}\phi=P^{\frac{1}{p}}(t)\mathbf{R}(Q^\frac{1}{q}(s)\psi),
\end{array} \right.
\end{equation}
where $(s,t)\in\R^N\times\R^N$ is regarded as parameters instead of independent variables.
The corresponding energy functional $J_{s,t}(z)$ given by 
$$J_{s,t}(z)=J_{s,t}(\psi,\phi)=\frac{1}{q^\prime}\int_{\R^N}|\psi|^{q^\prime}dx +\frac{1}{p^\prime}\int_{\R^N}|\phi|^{p^\prime}dx-\frac{1}{2}\int_{\R^N}z^T \mathbf{K}_{s,t} zdx,$$
where 
$$\mathbf{K}_{s,t} z=\left(\begin{array}{cc}
0 & \mathbf{K}_{q, p}^{Q(s), P(t)} \\
\mathbf{K}_{p, q}^{P(t), Q(s)} & 0
\end{array}\right)\binom{\psi}{\phi}.$$
 The Nehari set associated to $J_{s,t}$ is given by 
$$\mathcal{N}_{s,t}=\{z\in X\setminus\{0\}:\langle J_{s,t}^{\prime}(z),z \rangle
=0\},$$
and we define the ground energy level $c_{s,t}=\inf\limits_{z\in\mathcal{N}_{s,t} }J_{s,t}(z)$. Similarly, we know that $c_{s,t}=\inf\limits_{\gamma \in \Gamma} \max\limits_{t \in[0,1]} J_{s,t}(\gamma(t))>0$, where $\Gamma=\{\gamma \in C([0,1], X): \gamma(0)=0$ and $J_{s,t}(\gamma(1))<0\}$.
\begin{Lem}\label{Lem2.3}
Let $N\geq3$, $p, q>\frac{2 N}{N-1} \text { with } \frac{N-2}{N}<\frac{1}{p}+\frac{1}{q} <\frac{N-1}{N+1}$, then the system \eqref{eq2.8} admits a pair of nontrivial solution $(U_0,V_0)\in X$ such that $J_{s,t}(U_0,V_0)=c_{s,t}$. 
\end{Lem}
\begin{proof}
The proof of the existence of nontrivial critical point $\bar{z}=(U_0,V_0)\in \mathcal{N}_{s,t}$ can be found in \cite[Theorem 1.2]{Ding-Wang2023JDE}. Here, we only need to point out that the nontrivial critical point $\bar{z}$ is also a ground state solution of \eqref{eq2.8}. In fact, let $\{z_n\}\subset X$ is a $(PS)$-sequence for the functional $J_{s,t}$ at the mountain pass level $c_{s,t}$, we have $z_n\rightharpoonup \bar{z}$ weakly in $X$ and
$$c_{s,t}\leq J_{s,t}(\bar{z})\leq\liminf\limits_{n\rightarrow\infty}J_{s,t}(z_n)=c_{s,t}.$$
\end{proof}
\par
 To relate the solutions of  \eqref{eq2.4} to the set $M$, consider the limit problem
\begin{equation}\label{eq2.9}
\left\{\begin{array}{l}
|\psi|^{q^\prime-2}\psi=\bar{Q}^\frac{1}{q}\mathbf{R}(\bar{P}^\frac{1}{p}\phi), \\
|\phi|^{p^\prime-2}\phi=\bar{P}^\frac{1}{p}\mathbf{R}(\bar{Q}^\frac{1}{q}\psi).
\end{array} \right.
\end{equation}
Let $J_0$ be the energy functional of the system \eqref{eq2.9} and $\mathcal{N}_0$ be the associated Nehari set, i.e.,
$$J_0(z)=\frac{1}{q^\prime}\int_{\R^N}|\psi|^{q^\prime}dx +\frac{1}{p^\prime}\int_{\R^N}|\phi|^{p^\prime}dx-\frac{1}{2}\int_{\R^N}z^T \mathbf{K}_0 zdx$$
and 
$$\mathcal{N}_0=\{z\in X\setminus\{0\}:\langle J_0^{\prime}(z),z \rangle
=0\},\quad c_M=\inf_{z\in\mathcal{N}_0}J_0(z),$$
where 
$$\mathbf{K}_{0} z=\left(\begin{array}{cc}
0 & \mathbf{K}_{q, p}^{\bar{Q}, \bar{P})} \\
\mathbf{K}_{p, q}^{\bar{P}, \bar{Q}} & 0
\end{array}\right)\binom{\psi}{\phi}.$$
Let $\delta>0$ be fixed and $\eta$ be a smooth non-increasing function defined on $[0,+\infty)$, such that
$$\eta(t)= \begin{cases}1, & 0 \leq t \leq \frac{\delta}{2}, \\ 0, & t>\delta\end{cases}$$
with $0 \leq \eta(t) \leq 1,\left|\eta^{\prime}(t)\right| \leq c$. For any fixed $(x_0, x_0)\in M$ and $\varepsilon>0$, let us define 
\begin{equation}\label{eq++}
\left(\Psi_{\varepsilon, x_0}(x), \Phi_{\varepsilon, x_0}(x)\right):=\left(\eta\left(\left|\varepsilon x-x_0\right|\right) U\left(x-\varepsilon^{-1}x_0\right), \eta\left(\left|\varepsilon x-x_0\right|\right) V\left(x-\varepsilon^{-1}x_0\right)\right),
\end{equation}
where $z_0:=(U,V)\in X$, whose existence is guaranteed by Lemma \ref{Lem2.3}, is some fixed least-energy critical point of $J_0$.
\begin{Lem}\label{Lem2.4}
 There exists $t_{\varepsilon,x_0}>0$ such that 
$$\left(\psi_{\varepsilon, x_0}, \phi_{\varepsilon,x_0}\right)=\left(t_{\varepsilon, x_0} \bigg(\frac{Q(x_0)}{Q_\varepsilon}\bigg)^{\frac{1}{q}}\Psi_{\varepsilon, {x_0}}, t_{\varepsilon, x_0} \bigg(\frac{P(x_0)}{P_\varepsilon}\bigg)^{\frac{1}{p}}\Phi_{\varepsilon, {x_0}}\right) \in \mathcal{N}_\varepsilon$$
for $\varepsilon>0$ small.
\end{Lem}
\begin{proof}
It suffices to show that $\bar{z}_{\varepsilon,x_0}:=(\bar{\psi}_{\varepsilon,x_0},\bar{\phi}_{\varepsilon,x_0})=\bigg(\bigg(\displaystyle\frac{Q(x_0)}{Q_\varepsilon}\bigg)^{\frac{1}{q}}\Psi_{\varepsilon, {x_0}},\bigg(\frac{P(x_0)}{P_\varepsilon}\bigg)^{\frac{1}{p}}\Phi_{\varepsilon, {x_0}}\bigg)\in U_\varepsilon^+$. We start by remarking that 
\begin{equation}\label{eq2.14}
0<\|U\|^{q^{\prime}}_{q^{\prime}}+\|V\|^{p^{\prime}}_{p^{\prime}}=2\int_{\mathbb{R}^N} Q^{\frac{1}{q}}(x_0)U \mathbf{R}(P^{\frac{1}{p}}(x_0)V) dx,
\end{equation}
since $(U,V)$ is a nontrivial ground state solution of the system \eqref{eq2.9}. 
By the definition of $\bar{z}_{\varepsilon,x_0}$ and \eqref{eq2.14}, we know that 
\begin{equation}\label{eq2.15}
\begin{split}
\int_{\R^N}\bar{z}_{\varepsilon,x_0}^T \mathbf{K}_{\varepsilon} \bar{z}_{\varepsilon,x_0} dx&=2\int_{\mathbb{R}^N}Q^{\frac{1}{q}}(x_0)\Psi_{\varepsilon, x_0}\mathbf{R}(P^{\frac{1}{p}}(x_0)\Phi_{\varepsilon, x_0})dx\\
&=2 \int_{\mathbb{R}^N} Q^{\frac{1}{q}}(x_0) \eta(|\varepsilon x|)U(x)\mathbf{R}(P^{\frac{1}{p}}(x_0) \eta(|\varepsilon x|)V(x))dx\\
&\rightarrow 2\int_{\mathbb{R}^N} Q^{\frac{1}{q}}(x_0)U \mathbf{R}(P^{\frac{1}{p}}(x_0)V) dx>0\quad\text{as }\ \varepsilon\rightarrow0^+,
\end{split}
\end{equation}
where the last limit holds as a consequence of $ Q^{\frac{1}{q}}(x_0)\eta(|\varepsilon x|)U(x)\rightarrow Q^{\frac{1}{q}}(x_0)U(x)$ in $L^{q^\prime}(\mathbb{R}^N)$ and $P^{\frac{1}{p}}(x_0)\eta(|\varepsilon x|)V(x)\rightarrow P^{\frac{1}{p}}(x_0)V(x)$ in $L^{p^\prime}(\mathbb{R}^N)$ as $\varepsilon\rightarrow0^+$, together with the continuity of the linear integral operator $\mathbf{R}$.
Moreover, the limit is uniform in $(x_0,x_0)\in M$, because $M$ is a compact set. Therefore, $\bar{z}_{\varepsilon,x_0}\in U_\varepsilon^+$ for all $(x_0,x_0)\in M $ and $\varepsilon>0$ small enough. 
\hfill
\end{proof}
\begin{Lem}\label{Lem2.5}
$$\lim _{\varepsilon \rightarrow 0^{+}} J_{\varepsilon}(\psi_{\varepsilon, x_0}, \phi_{\varepsilon,x_0})=\lim _{\varepsilon \rightarrow 0^{+}} J_{\varepsilon}(t_{\varepsilon, x_0} \bar{z}_{\varepsilon,x_0})=c_M \text {, uniformly for }(x_0,x_0) \in M .$$
\end{Lem}
\begin{proof}
Firstly, we prove that $t_{\varepsilon, x_0}\rightarrow1$ as $\varepsilon\rightarrow0^+$, uniformly for $(x_0,x_0)\in M$. Indeed, given any sequences $\varepsilon_n\rightarrow0$ as $n\rightarrow\infty$, and for any $(x_0,x_0)\in M$, we have
\begin{equation}\label{eq2.16}
\int_{\R^N}|\bar{\psi}_{\varepsilon_n, x_0}|^{q^{\prime}}dx=\int_{\R^N}\bigg|\bigg(\frac{Q(x_0)}{Q(\varepsilon_n x+x_0)}\bigg)^{\frac{1}{q}}\eta(|\varepsilon_n x|)U(x)\bigg|^{q^{\prime}}dx\rightarrow\int_{\R^N}|U|^{q^{\prime}}dx,
\end{equation}
as $n\rightarrow\infty$. Similarly, $\int_{\R^N}|\bar{\phi}_{\varepsilon_n, x_0}|^{p^{\prime}}dx\rightarrow\int_{\R^N}|V|^{p^{\prime}}dx$. Thus, it follows from \eqref{eq2.15} and \eqref{eq2.16} that 
\begin{equation*}
\begin{split}
\langle J_{\varepsilon_n}^\prime(\bar{z}_{{\varepsilon_n}, x_0}),\bar{z}_{{\varepsilon_n}, x_0}^T \rangle=&\int_{\R^N}|\bar{\psi}_{\varepsilon_n, x_0}|^{q^{\prime}}dx+\int_{\R^N}|\bar{\phi}_{\varepsilon_n, x_0}|^{p^{\prime}}dx-\int_{\R^N}\bar{z}_{{\varepsilon_n}, x_0}^T  \mathbf{K}_{\varepsilon_n} \bar{z}_{\varepsilon_n,x_0} dx\\
\rightarrow&\int_{\R^N}|U|^{q^{\prime}}dx+\int_{\R^N}|V|^{p^{\prime}}dx\\
&-\int_{\mathbb{R}^N}Q^{\frac{1}{q}}(x_0)U \mathbf{R}(P^{\frac{1}{p}}(x_0)V) dx-\int_{\mathbb{R}^N}P^{\frac{1}{p}}(x_0)V \mathbf{R}(Q^{\frac{1}{q}}(x_0)U) dx\\
=&\langle J_{0}^\prime(z_0),z_0^T\rangle=0, \quad \text{as}\  n\rightarrow\infty,
\end{split}
\end{equation*}
uniformly for $(x_0,x_0)\in M $, where $z_0=(U,V)$ is the least-energy critical point of $J_0$. Consequently,
\begin{equation}\label{eq2.17}
\int_{\R^N}|\bar{\psi}_{\varepsilon_n, x_0}|^{q^{\prime}}dx+\int_{\R^N}|\bar{\phi}_{\varepsilon_n, x_0}|^{p^{\prime}}dx=\int_{\R^N}\bar{z}_{{\varepsilon_n}, x_0}^T  \mathbf{K}_{\varepsilon_n} \bar{z}_{\varepsilon_n,x_0} dx+o_n(1).
\end{equation}
Using the fact that $t_{\varepsilon_n, x_0}\bar{z}_{{\varepsilon_n}, x_0}\in \mathcal{N}_{\varepsilon_n}$, we get 
\begin{equation}\label{eq2.18}
t_{\varepsilon_n, x_0}^{q^{\prime}-2}\int_{\R^N}|\bar{\psi}_{\varepsilon_n, x_0}|^{q^{\prime}}dx+t_{\varepsilon_n, x_0}^{p^{\prime}-2}\int_{\R^N}|\bar{\phi}_{\varepsilon_n, x_0}|^{p^{\prime}}dx=\int_{\R^N}\bar{z}_{{\varepsilon_n}, x_0}^T  \mathbf{K}_{\varepsilon_n} \bar{z}_{\varepsilon_n,x_0} dx.
\end{equation}
From \eqref{eq2.17} and \eqref{eq2.18}, we have 
\begin{equation}\label{eq2.19}
(t_{\varepsilon_n, x_0}^{q^{\prime}-2}-1)\int_{\R^N}|\bar{\psi}_{\varepsilon_n, x_0}|^{q^{\prime}}dx+(t_{\varepsilon_n, x_0}^{p^{\prime}-2}-1)\int_{\R^N}|\bar{\phi}_{\varepsilon_n, x_0}|^{p^{\prime}}dx=o_n(1).
\end{equation}
Taking the limit in \eqref{eq2.19}, we obtain $t_{\varepsilon_n, x_0}\rightarrow1$ as $n\rightarrow\infty$, uniformly for $(x_0,x_0)\in M$, and we obtain 
\begin{equation*}
\begin{split}
J_\varepsilon(\psi_{\varepsilon,x_0},\phi_{\varepsilon,x_0})&=\bigg(\frac{1}{q^\prime}-\frac{1}{2}\bigg)t_{\varepsilon,x_0}^{q^\prime}\int_{\mathbb{R}^N}\left|\bar{\psi}_{\varepsilon, x_0}\right|^{q^{\prime}} dx+\bigg(\frac{1}{p^\prime}-\frac{1}{2}\bigg)t_{\varepsilon,x_0}^{p^\prime}\int_{\mathbb{R}^N}\left|\bar{\phi}_{\varepsilon, x_0}\right|^{p^{\prime}} dx\\
&\rightarrow\bigg(\frac{1}{q^\prime}-\frac{1}{2}\bigg)\int_{\mathbb{R}^N}|U|^{q^{\prime}}dx+\bigg(\frac{1}{p^\prime}-\frac{1}{2}\bigg)\int_{\mathbb{R}^N}|V|^{p^{\prime}}dx\\
&=J_{0}(U,V)=J_{0}(z_0)=c_{M}\quad{as}\ \varepsilon\rightarrow0^+,
\end{split}
\end{equation*}
uniformly for $(x_0,x_0)\in M $.
\hfill
\end{proof}
\par
\begin{Lem}\label{Lem2.6}
For any $\varepsilon<0$, there holds $c_{M}\leq c_\varepsilon$ and $\lim\limits_{\varepsilon\rightarrow0^+}c_\varepsilon=c_{M}$. 
\end{Lem}
\begin{proof}
For any $z_\varepsilon=(\psi_\varepsilon,\phi_\varepsilon)\in \mathcal{N}_\varepsilon$ and set $w_0=(\psi_0,\phi_0)=((Q_\varepsilon/\bar{Q})^{\frac{1}{q}}\psi_\varepsilon,(P_\varepsilon/\bar{P})^{\frac{1}{p}}\phi_\varepsilon)$. It is easy to see that 
\begin{equation}\label{eq2.21}
\begin{split}
&\int_{\R^N}\bar{P}^{\frac{1}{p}}\phi_0\mathbf{R}(\bar{Q}^{\frac{1}{q}}\psi_0)dx=\int_{\R^N}\bar{Q}^{\frac{1}{q}}\psi_0\mathbf{R}(\bar{P}^{\frac{1}{p}}\phi_0)dx\\
=&\int_{\R^N}P_\varepsilon^{\frac{1}{p}}\phi_\varepsilon\mathbf{R}(Q_\varepsilon^{\frac{1}{q}}\psi_\varepsilon)dx=\int_{\R^N}Q_\varepsilon^{\frac{1}{q}}\psi_\varepsilon\mathbf{R}(P_\varepsilon^{\frac{1}{p}}\phi_\varepsilon)dx>0,
\end{split}
\end{equation}
which implies that $w_0\in U_0^+=\{z\in X:\int_{\R^N}z^T\mathbf{K}_0zdx>0\}$ and there exists $t_\varepsilon>0$ such that $t_\varepsilon w_0\in \mathcal{N}_0$. Then, we have 
\begin{equation*}
\begin{split}
\|\psi_\varepsilon\|_{q^\prime}^{q^\prime}+\|\phi_\varepsilon\|_{p^\prime}^{p^\prime}&=\int_{\R^N}z_\varepsilon^T \mathbf{K}_\varepsilon z_\varepsilon dx= \int_{\R^N}w_0^T \mathbf{K}_0 w_0 dx \\
&=t_\varepsilon^{q^\prime-2}\|\psi_0\|_{q^\prime}^{q^\prime}+t_\varepsilon^{p^\prime-2}\|\phi_0\|_{p^\prime}^{p^\prime}\\
&\leq t_\varepsilon^{q^\prime-2}\|\psi_\varepsilon\|_{q^\prime}^{q^\prime}+t_\varepsilon^{p^\prime-2}\|\phi_\varepsilon\|_{p^\prime}^{p^\prime}.
\end{split}
\end{equation*}
and we conclude that $t_\varepsilon\leq 1$. 
As a consequence, we obtain
\begin{equation*}
\begin{split}
c_M&\leq J_0(t_\varepsilon z_0)=(\frac{1}{q^\prime}-\frac{1}{2})t_\varepsilon^{q^\prime}\|\psi_0\|_{q^\prime}^{q^\prime}+(\frac{1}{p^\prime}-\frac{1}{2})t_\varepsilon^{p^\prime}\|\phi_0\|_{p^\prime}^{p^\prime}\\
&\leq(\frac{1}{q^\prime}-\frac{1}{2})\|\psi_\varepsilon\|_{q^\prime}^{q^\prime}+(\frac{1}{p^\prime}-\frac{1}{2})\|\phi_\varepsilon\|_{p^\prime}^{p^\prime}=J_\varepsilon(z_\varepsilon),
\end{split}
\end{equation*}
Using the arbitrariness of $z_\varepsilon$, there holds $c_M\leq c_\varepsilon$. Lemma \ref{Lem2.5} gives for $(x_0,x_0) \in M$, $c_\varepsilon\leq J_{\varepsilon}(\psi_{\varepsilon, x_0}, \phi_{\varepsilon,x_0})\rightarrow c_M$ as $\varepsilon\rightarrow0^+$, combining the previous inequality, we have $\lim\limits_{\varepsilon\rightarrow0^+}c_\varepsilon=c_M$.
\hfill
\end{proof}
\section{Palais-Smale condition and representation Lemma}
To establish the Palais-Smale condition for $J_\varepsilon$ on $\mathcal{N_\varepsilon}$ below the level $c_\infty$, we first prove the following decay estimation for the nonlocal interaction induced by the generalized Birman-Schwinger operator. The following Lemma can also be regarded as a generalization of \cite[Lemma 2.4]{Evequoz2017AMPA}.
\begin{Lem}\label{Lem3.1}
Let $N \geq 3$, $\frac{2 N}{N-1}<p,q \leq \infty$, and $\frac{1}{q}+\frac{1}{p}<\frac{N-1}{N+1}$. Put $M_R:=\mathbb{R}^N \backslash B_R$ for $R>0$,
then there exists a constant $C=C(N,p,q)>0$ such that for any $R>0$, $m\geq8$ and $u\in L^{q^{\prime}}(\R^N),v\in L^{p^{\prime}}(\R^N)$ with $supp(u)\subset  B_{R}$ and $supp(v)\subset  M_{R+m}$,
$$\left|\int_{\mathbb{R}^N} u \mathbf{R} v \mathrm{~d} x\right| \leq C m^{-\lambda}\|u\|_{q^{\prime}}\|v\|_{p^{\prime}},$$
 where  
 $$\lambda=\lambda(p, q):=-\max \left\{\frac{N}{p}+\frac{1-N}{2}, \frac{N}{q}+\frac{1-N}{2}, \frac{1-N}{2}+
 \frac{N+1}{2}\left(\frac{1}{q}+\frac{1}{p}\right)\right\}>0.$$
\end{Lem}
\begin{proof}
Let $u,v\in \mathscr{S}$ with $supp(u)\subset B_{R}$ and $supp(v) \subset M_{R+m} $, by the symmetry of the operator $\mathcal{R}$ we know that
\begin{equation}\label{eq3.1}
\begin{split}
\left|\int_{\mathbb{R}^N} u \mathcal{R} v dx\right|=\left|\int_{\mathbb{R}^N} v \mathcal{R} u dx\right| &\leq\|v\|_{p^{\prime}}\|\Phi * u\|_{L^p\left(M_{R+m}\right)}\\
&\leq\|v\|_{p^{\prime}}(\|\Phi_1* u\|_{L^p(M_{R+m})}+\|\Phi_2* u\|_{L^p(M_{R+m})}),
\end{split}
\end{equation}
where the decomposition of the fundamental solution $\Phi$ proceeds as follows: fix $\psi\in\mathscr{S}$ such that $\widehat{\psi}\in\mathcal{C}_c^{\infty}\left(\mathbb{R}^N\right)$ is radial, $0 \leq \widehat{\psi} \leq 1, \widehat{\psi}(\xi)=1$ for $\| \xi|-1| \leq \frac{1}{6}$ and $\widehat{\psi}(\xi)=0$ for $||\xi|-1| \geq \frac{1}{4}$, we can write $\Phi=\Phi_1+\Phi_2$ with
$$\Phi_1:=(2 \pi)^{-\frac{N}{2}}(\psi * \Phi), \quad \Phi_2:=\Phi-\Phi_1,$$
and there are the following estimates (see \cite{Evequoz-Weth2015AM}):
\begin{equation}\label{eq3.2}
\begin{split}
\left|\Phi_1(x)\right|&\leq C_0(1+|x|)^{\frac{1-N}{2}} \quad\quad\quad\quad\  \text { for } x \in \mathbb{R}^N;\\
\left|\Phi_2(x)\right|&\leq C_0\min\{|x|^{2-N},|x|^{-N}\}\quad \text { for }\  x \in \mathbb{R}^N\setminus\{0\}.
\end{split}
\end{equation}
Firstly, we provide an estimate for the item $\|\Phi_2* u\|_{L^p(M_{R+m})}$. By \eqref{eq3.2} and Young's inequality, we have
\begin{equation}\label{eq3.3}
\begin{split}
\|\Phi_2* u\|_{L^p(M_{R+m})}&\leq\|[1_{M_m}|\Phi_2|]\ast|u|\|_p\leq \|1_{M_m}\Phi_2\|_r\|u\|_{q^\prime}\\
&\leq C \bigg(\int_{|x|\geq m}|x|^{-Nr}dx\bigg)^{\frac{1}{r}}\|u\|_{q^\prime}\\
&=C m^{\frac{N-Nr}{r}}\|u\|_{q^\prime},
\end{split}
\end{equation}
where $\frac{1}{r}=\frac{1}{p}+\frac{1}{q}$. To estimate $\|\Phi_1* u\|_{L^p(M_{R+m})}$, as in \cite[Lemma 2.4]{Evequoz2017AMPA}, we fix a radial function $\phi \in \mathscr{S}\left(\mathbb{R}^N\right)$ such that $\widehat{\phi} \in \mathcal{C}_c^{\infty}\left(\mathbb{R}^N\right)$ is radial, $0 \leq \widehat{\phi} \leq 1, \widehat{\phi}(\xi)=1$ for $\| \xi|-1| \leq \frac{1}{4}$ and $\widehat{\phi}(\xi)=0$ for $||\xi|-1| \geq \frac{1}{2}$. Moreover, let $\tilde{u}:=\phi * u \in \mathscr{S}\left(\mathbb{R}^N\right)$, then $\Phi_1 * u=(2 \pi)^{-\frac{N}{2}}\left(\Phi_1 * \tilde{u}\right)$, and so 
\begin{equation}\label{eq3.4}
\begin{split}
\|\Phi_1* u\|_{L^p(M_{R+m})}&=C\|\Phi_1 * \tilde{u}\|_{L^p(M_{R+m})}\\
&\leq C(\|[1_{B_{\frac{m}{2}}} \Phi_1] * \tilde{u}\|_{L^p(M_{R+m})}+\|[1_{M_{\frac{m}{2}}} \Phi_1] * \tilde{u}\|_{L^p(M_{R+m})}).
\end{split}
\end{equation}
By the \cite[Corollary 2.11]{Ding-Wang2023JDE} and Young's inequality, we know that 
\begin{equation}\label{eq3.5}
\begin{split}
\|[1_{M_{\frac{m}{2}}} \Phi_1] * \tilde{u}\|_{L^p(M_{R+m})}&\leq C\bigg(\frac{m}{2}\bigg)^{-\lambda}\|\tilde{u}\|_{q^\prime}=C\bigg(\frac{m}{2}\bigg)^{-\lambda}\|\phi * u\|_{q^\prime}\leq C\bigg(\frac{m}{2}\bigg)^{-\lambda}\|\phi\|_{1}\|u\|_{q^\prime}.
\end{split}
\end{equation}
A similar estimation of \eqref{eq3.3} yields
\begin{equation*}
\begin{split}
\|[1_{B_{\frac{m}{2}}} \Phi_1] * \tilde{u}\|_{L^p(M_{R+m})}&=\|([1_{B_{\frac{m}{2}}} \Phi_1] * \phi) *u\|_{L^p(M_{R+m})}\\
&\leq \|1_{M_m}|\left[1_{B \frac{m}{2}} \Phi_1\right] * \phi |* |u|\|_p\\
&\leq \|1_{M_m}(\left[1_{B \frac{m}{2}} \Phi_1\right] * \phi )\|_r\|u\|_{q^\prime}.
\end{split}
\end{equation*}
Note that $\phi\in \mathscr{S}$ and combine it with \eqref{eq3.2} to know
 \begin{equation*}
\begin{split}
\|1_{M_m}(\left[1_{B \frac{m}{2}} \Phi_1\right] * \phi )\|^r_r&=\int_{|x|\geq m}\bigg|\int_{|y|<\frac{m}{2}}\Phi_1(y)\phi(x-y)dy\bigg|^rdx\\
&\leq \int_{|x|\geq m}\bigg(\int_{|y|<\frac{m}{2}}C_0(1+|y|)^{\frac{1-N}{2}}|\phi(x-y)|dy\bigg)^rdx\\
&\leq C \int_{|x|\geq m}\bigg(\int_{|y|<\frac{m}{2}}|x-y|^{-s}dy\bigg)^rdx\\
&\leq C |B_{\frac{m}{2}}|^r \int_{|x|\geq m}\bigg (|x|-\frac{m}{2}\bigg)^{-sr}dx\\
&=\frac{C\omega_N}{N}\bigg(\frac{m}{2}\bigg)^{Nr+N-sr}\int_{|z|\geq 1}\bigg(|z|-\frac{1}{2}\bigg)^{-sr}dx\\
&=C m^{Nr+N-sr},
\end{split}
\end{equation*}
where $s$ may be fixed so large that $s\geq 2N$. Consequently, 
\begin{equation}\label{eq3.6}
\|[1_{B_{\frac{m}{2}}} \Phi_1] * \tilde{u}\|_{L^p(M_{R+m})}\leq Cm^{\frac{Nr+N-sr}{r}}\|u\|_{q^\prime}.
\end{equation}
 It is easy to see that the function $k(t)=m^t$ increases with respect to $t\in \R$ since $m\geq8$, and 
$$-\lambda\geq \frac{N}{q}+\frac{1-N}{2}\geq \frac{N-N r}{r}\geq\frac{N r+N-s r}{r}.$$
Thus, for any $u,v\in \mathscr{S}$ with  $supp(u)\subset  B_{R}$ and $supp(v)\subset  M_{R+m}$, it holds by \eqref{eq3.1}, \eqref{eq3.5} and \eqref{eq3.6} that
\begin{equation}\label{eq3.7}
\left|\int_{\mathbb{R}^N} u \mathcal{R} v dx\right|=\left|\int_{\mathbb{R}^N} v \mathcal{R} u dx\right|\leq C m^{-\lambda}\|u\|_{q^{\prime}}\|v\|_{p^{\prime}}.
\end{equation}
Finally, the proof is completed via the density argument.
\hfill
\end{proof}

Let $c_\infty$ be the ground energy associated with 
\begin{equation}\label{eq3.8}
\left\{\begin{array}{l}
|\psi|^{q^\prime-2}\psi=Q_\infty^{\frac{1}{q}}\mathbf{R}(P_\infty^\frac{1}{p}\phi), \\
|\phi|^{p^\prime-2}\phi=P_\infty^{\frac{1}{p}}\mathbf{R}(Q_\infty^\frac{1}{q}\psi),
\end{array} \right.
\end{equation}
we have $c_\infty=\inf\limits_{z\in\mathcal{N}_\infty}J_\infty(z)>0$, where $\mathcal{N}_\infty$ and $J_\infty$ represent the corresponding Nehari set and energy functional associated with \eqref{eq3.8}. For convenience, we introduce the notation $\mathbf{K}_\infty$ as follows
$$\mathbf{K}_{\infty} z=\left(\begin{array}{cc}
0 & \mathbf{K}_{q, p}^{Q_\infty, P_\infty)} \\
\mathbf{K}_{p, q}^{P_\infty, Q_\infty} & 0
\end{array}\right)\binom{\psi}{\phi}\quad \text{for } z=(\psi,\phi)^T\in X.$$
Next, we prove the Palais-Smale condition for the dual energy functional $J_\varepsilon$ on the Nehari manifold $\mathcal{N}_\varepsilon$ below the limit energy level $c_\infty$.

\begin{Lem}\label{Lem3.2}
For $\varepsilon>0$ sufficiently small, If $(PQ_1)$$-$$(PQ_3)$ and $c_\varepsilon<c_\infty$ hold, then $J_\varepsilon$ satisfies the Palais-Smale condition in $\{z\in \mathcal{N}_\varepsilon:J_\varepsilon< c_\infty\}$, i.e., every sequence $\{z_n\}\subset \mathcal{N}_\varepsilon$ such that $J_\varepsilon(z_n)\rightarrow d<c_\infty$ and $J^\prime_\varepsilon|_{\mathcal{N}_\varepsilon}(z_n)\rightarrow0$ as $n\rightarrow\infty$ has a convergent subsequence.
\end{Lem}
\begin{proof}
To begin with, by $c_M<c_\infty$ and Lemma \ref{Lem2.5}, we know that the sublevel $\{z\in \mathcal{N}_\varepsilon:J_\varepsilon< c_\infty\}$ is not empty for $\varepsilon>0$ small enough. Next, we just need to consider $d\in[c_\varepsilon,c_\infty)$ and let $\{z_n\}=\{(\psi_n,\phi_n)^T\}$ be a $(PS)_d$-sequence for $J_\varepsilon$ on $\mathcal{N}_\varepsilon$. 
We can assume that $\{z_n\}$ is a $(PS)_d$-sequence for the unconstrained functional $J_\varepsilon$ because $\mathcal{N}_\varepsilon$ is a natural constraint. By Lemma \ref{Lem2.2+}, up to a subsequence, there exists $z=(\psi,\phi)^T\in X$ such that $z_n\rightharpoonup z$ in $X$, which implies that $\psi_n \rightharpoonup \psi \in L^{q^{\prime}}\left(\mathbb{R}^N\right)$ and $\phi_n \rightharpoonup \phi \in L^{p^{\prime}}\left(\mathbb{R}^N\right)$. In addition, we obtain that $\psi_n \rightarrow \psi \in L^{q^{\prime}}\left(B_R\right)$ and $\phi_n \rightarrow \phi \in L^{p^{\prime}}\left(B_R\right)$ via Sobolev's embedding theorem for any $R>0$, more details can be found in \cite[Lemma 2.6 (b)]{Ding-Wang2023JDE}. In order to prove $z_n\rightarrow z$ in $X$, it suffices to show that for any $\delta>0$ there exists $R>0$ such that 
\begin{equation}\label{eq3.10}
\int_{|x| \geq R}\left|\psi_n\right|^{q^\prime} d x<\delta, \quad \int_{|x| \geq R}\left|\phi_n\right|^{p^\prime} d x<\delta, \quad \forall n.
\end{equation}

For this we first claim that: $\forall \eta>0 \ \text{and}\  \forall R>0, \exists r>R \ \text { such that}$
\begin{equation}\label{eq3.11}
\liminf\limits _{n \rightarrow \infty}\int_{r<|x|<2r}\left|\psi_n\right|^{q^{\prime}} dx<\eta\quad \text{and}\quad \liminf\limits _{n \rightarrow \infty}\int_{r<|x|<2r}\left|\phi_n\right|^{p^{\prime}} dx<\eta.
\end{equation}
In fact, if not, there exists $\eta_0, R_0>0$ and $n_0=n_0(m)$ for every $m>R_0$, such that 
\begin{equation}\label{eq3.12}
\int_{m<|x|<2m}\left|\psi_n\right|^{q^{\prime}} dx\geq\eta_0\quad \text{or}\quad \int_{m<|x|<2m}\left|\phi_n\right|^{p^{\prime}} dx\geq\eta_0,\ \forall n\geq n_0.
\end{equation}
Without loss of generality, we assume that $n_0(m+1)\geq n_0(m)$ for all $m$ and the first inequality in \eqref{eq3.12} holds. Thus, for every $ l\in\mathbb{N}$, there exists $N_0(l)$ such that 
$$
\int_{\mathbb{R}^N}\left|\psi_n\right|^{q^{\prime}} \mathrm{d} x \geq \sum_{k=0}^{l-1} \int_{2^k\left(\left[R_0\right]+1\right)<|x|<2^{k+1}\left(\left[R_0\right]+1\right)}\left|\psi_n\right|^{q^{\prime}} \mathrm{d} x \geq l \eta_0\rightarrow\infty
$$
if $l\rightarrow\infty$, $\forall n\geq N_0(l)$. This contradicts the fact that $\|\psi_n\|_{q\prime}\leq\|z_n\|\leq C_1$.

Returning to inequalities \eqref{eq3.10}, we assume that it does not hold, then there exists a subsequence $\{z_{n_k}\}=\{(\psi_{n_k},\phi_{n_k})^T\}\subset X$ and a $\delta_0>0$ such that 
\begin{equation}\label{eq3.13}
\int_{|x|\geq k}\left|\psi_{n_k}\right|^{q^{\prime}} dx\geq\delta_0\quad \text{or}\quad \int_{|x|\geq k}\left|\phi_{n_k}\right|^{p^{\prime}} dx\geq\delta_0, \ \forall k .
\end{equation}
Now, let us fix $0<\eta<\min\{1, (\frac{\delta_0}{6 C_2})^{p^{\prime}}\}$, where $C_2=\max\{C_1^2,C_1\}C(N,p,q)\|P\|^{\frac{1}{p}}_\infty \|Q\|^{\frac{1}{q}}_\infty$. Here, the constant $C(N,p,q)>0$ is chosen such that Lemma \ref{Lem3.1} holds and $\|\mathbf{R}(\psi)\|_{p}\leq C(N,p,q)\|\psi\|_{q^\prime}$ for any $\psi\in L^{q^\prime}(\mathbb{R}^N)$. In addition, by $(PQ)_2$ we can find a $R(\eta)>0$ such that 
\begin{equation}\label{eq3.14}
Q_{\varepsilon}(x) \leq Q_{\infty}+\eta\quad \text{and}\quad P_{\varepsilon}(x) \leq P_{\infty}+\eta\ \quad\text{for}\  |x| \geq R(\eta).
\end{equation}
Due to \eqref{eq3.11}, we can find  $r>\max\{8, R(\eta), \eta^{-\frac{1}{\lambda p^\prime }},\eta^{-\frac{1}{\lambda q^\prime }}\}$ and subsequence which still denoted by $\{z_{n_k}\}$, such that
\begin{equation}\label{eq3.15}
\int_{r<|x|<2 r}\left|\psi_{n_k}\right|^{q^{\prime}} d x<\eta\quad \text{and}\quad \int_{r<|x|<2 r}\left|\phi_{n_k}\right|^{p^{\prime}} d x <\eta,\quad  \forall k .
\end{equation}
Define $w_{n_k}:=1_{\{|x|\geq2r\}}z_{n_k}=(\hat{\psi}_{n_k},\hat{\phi}_{n_k})^T$, from Lemma \ref{Lem3.1}, \eqref{eq3.15} and H\"{o}lder inequality, we have 
\begin{equation}\label{eq3.16}
\begin{split}
&\bigg|\langle J^\prime_\varepsilon(z_{n_k}),(\hat{\psi}_{n_k},0)^T \rangle-\langle J^\prime_\varepsilon(w_{n_k}),(\hat{\psi}_{n_k},0)^T \rangle\bigg|
=\bigg|\int_{|x|<2r}P_\varepsilon^{\frac{1}{p}}\phi_{n_k}\mathbf{R}(Q_\varepsilon^{\frac{1}{q}}\hat{\psi}_{n_k})dx\bigg|\\
\leq & \bigg|\int_{|x|<r}P_\varepsilon^{\frac{1}{p}}\phi_{n_k}\mathbf{R}(Q_\varepsilon^{\frac{1}{q}}\hat{\psi}_{n_k})dx\bigg|+\bigg|\int_{r<|x|<2r}P_\varepsilon^{\frac{1}{p}}\phi_{n_k}\mathbf{R}(Q_\varepsilon^{\frac{1}{q}}\hat{\psi}_{n_k})dx\bigg|\\
\leq & C(N,p,q)C_1^2\|P\|^{\frac{1}{p}}_\infty \|Q\|^{\frac{1}{q}}_\infty r^{-\lambda}+C(N,p,q)\|P\|_\infty^{\frac{1}{p}}\|Q\|_\infty^{\frac{1}{q}}\|\psi_{n_k}\|_{q^\prime}\bigg(\int_{r<|x|<2r}|\phi_{n_k}|^{p^\prime}dx\bigg)^{\frac{1}{p^\prime}}\\
\leq &C(N,p,q)C_1^2\|P\|^{\frac{1}{p}}_\infty \|Q\|^{\frac{1}{q}}_\infty r^{-\lambda}+C(N,p,q)C_1\|P\|^{\frac{1}{p}}_\infty \|Q\|^{\frac{1}{q}}_\infty\eta^{\frac{1}{p\prime}}\\
\leq & C_2(r^{-\lambda}+\eta^{\frac{1}{p\prime}})\\
\leq& 2C_2  \eta^{\frac{1}{p\prime}}.
\end{split}
\end{equation}
Similarly, we obtain that 
$$|\langle J^\prime_\varepsilon(z_{n_k}),(0,\hat{\phi}_{n_k})^T \rangle-\langle J^\prime_\varepsilon(w_{n_k}),(0,\hat{\phi}_{n_k})^T \rangle|\leq 2C_2  \eta^{\frac{1}{q\prime}}.$$
Without loss of generality, by \eqref{eq3.13} we suppose that 
\begin{equation}\label{eq3.17}
\int_{\R^N}|\hat{\psi}_{n_k}|^{q^\prime}dx=\int_{\{|x|>2 r\}}|\psi_{n_k}|^{q^\prime}dx\geq \delta_0 \quad \text{for all}\  k\geq 2r.
\end{equation}
Note that the choice of $\eta$ and the fact that $\langle J_\varepsilon^\prime(z_{n_k}),(\hat{\psi}_{n_k},0)^T \rangle\rightarrow0$ as $k\rightarrow\infty$, when combined with \eqref{eq3.16} and \eqref{eq3.17}, allow us to find a $k_0=k_0(r,\eta,\delta_0)\geq 2r$ such that
\begin{equation}\label{eq3.18}
\begin{split}
\int_{\R^N}\hat{\psi}_{n_k}\mathbf{K}_{q, p}^{Q_\varepsilon, P_{\varepsilon}}\hat{\phi}_{n_k}dx=&\int_{\R^N}|\hat{\psi}_{n_k}|^{q^\prime}dx-\langle J_\varepsilon^\prime(w_{n_k}),(\hat{\psi}_{n_k},0)^T\rangle\\
=&\int_{\R^N}|\hat{\psi}_{n_k}|^{q^\prime}dx-\langle J_\varepsilon^\prime(z_{n_k}),(\hat{\psi}_{n_k},0)^T\rangle\\
&+[\langle J_\varepsilon^\prime(z_{n_k}),(\hat{\psi}_{n_k},0)^T\rangle-\langle J_\varepsilon^\prime(w_{n_k}),(\hat{\psi}_{n_k},0)^T\rangle]\\
\geq&\int_{\R^N}|\hat{\psi}_{n_k}|^{q^\prime}dx-|\langle J_\varepsilon^\prime(z_{n_k}),(\hat{\psi}_{n_k},0)^T\rangle|-2C_2\eta^{\frac{1}{p\prime}}\\
\geq&\frac{\delta_0}{3}, \quad \text{for all} \  k\geq k_0,
\end{split}
\end{equation}
which implies that 
\begin{equation}\label{eq3.19}
\int_{\R^N}w_{n_k}\mathbf{K}_{\varepsilon}w_{n_k}dx=2\int_{\R^N}\hat{\psi}_{n_k}\mathbf{K}_{q, p}^{Q_\varepsilon, P_{\varepsilon}}(\hat{\phi}_{n_k})dx\geq\frac{2\delta_0}{3}>0, \quad \text{for all} \  k\geq k_0.
\end{equation}
It follows from $1<p^{\prime}, q^{\prime}<2$ that 
\begin{equation}\label{eq3.20}
\left(\frac{1}{q^{\prime}}-\frac{1}{2}\right)\int_{\R^N}|\hat{\psi}_{n_k}|^{q^\prime}dx+\left(\frac{1}{p^{\prime}}-\frac{1}{2}\right)\int_{\R^N}|\hat{\phi}_{n_k}|^{p^\prime}dx\leq J_\varepsilon(z_{n_k}).
\end{equation}
 For $ k \geq k_0$, let $ \tilde{w}_k:=((\frac{Q_{\varepsilon}}{Q_{\infty}})^{\frac{1}{q}}\hat{\psi}_{n_k},(\frac{P_{\varepsilon}}{P_{\infty}})^{\frac{1}{p}}\hat{\phi}_{n_k})$, using \eqref{eq3.18} we have
 $$\int_{\R^N}\tilde{w}_k\mathbf{K}_{\infty}\tilde{w}_kdx=\int_{\R^N}w_{n_k}\mathbf{K}_{\varepsilon}w_{n_k}dx\geq\frac{2\delta_0}{3}>0,$$ 
which implies that there is a $t_{k,\infty}>0$ such that $t_{k,\infty}\tilde{w}_k\in \mathcal{N}_\infty $. In order to obtain $t_{k,\infty}\leq1$, for any fixed $z_{n_k}\in \mathcal{N}_\varepsilon$, we consider the function 
$$h(t):=t^{q^\prime-2}\int_{\R^N}\bigg|\bigg(\frac{Q_\varepsilon}{Q_\infty}\bigg)^{\frac{1}{q}}\hat{\psi}_{n_k}\bigg|^{q^\prime}dx+t^{p^\prime-2}\int_{\R^N}\bigg|\bigg(\frac{P_\varepsilon}{P_\infty}\bigg)^{\frac{1}{p}}\hat{\phi}_{n_k}\bigg|^{p^\prime}dx-\int_{\R^N}\tilde{w}_k\mathbf{K}_{\infty}\tilde{w}_kdx,$$
where $1<q^\prime,p^\prime<2$. We claim that $h(t)$ is a strictly decreasing function of $t$. In fact, for any fixed $z_{n_k}\in\mathcal{N}_\varepsilon$, we have $z_{n_k}\neq0$ and 
\begin{equation*}
\begin{split}
h^\prime(t)&=(q^\prime-2)t^{q^\prime-3}\int_{\mathbb{R}^N}\bigg|\bigg(\frac{Q_\varepsilon}{Q_\infty}\bigg)^{\frac{1}{q}}\hat{\psi}_{n_k}\bigg|^{q^\prime}dx+(p^\prime-2)t^{p^\prime-3}\int_{\mathbb{R}^N}\bigg|\bigg(\frac{P_\varepsilon}{P_\infty}\bigg)^{\frac{1}{p}}\hat{\phi}_{n_k}\bigg|^{p^\prime}dx\\
&\leq(q^\prime-2)t^{q^\prime-3}\bigg(\frac{\alpha}{Q_\infty}\bigg)^\frac{q^\prime}{q}\int_{\mathbb{R}^N}|\hat{\psi}_{n_k}|^{q^\prime}dx+(p^\prime-2)t^{p^\prime-3}\bigg(\frac{\alpha}{P_\infty}\bigg)^\frac{p^\prime}{p}\int_{\mathbb{R}^N}|\hat{\phi}_{n_k}|^{p^\prime}dx\\
&<0, \quad \forall\ t>0. 
\end{split}
\end{equation*}
Using the fact that $\langle J^\prime_\varepsilon(z_{n_k}),z_{n_k}\rangle=0$, we get
\begin{equation*}
\begin{split}
h(1)=&\int_{\R^N}\bigg|\bigg(\frac{Q_\varepsilon}{Q_\infty}\bigg)^{\frac{1}{q}}\hat{\psi}_{n_k}\bigg|^{q^\prime}dx+\int_{\R^N}\bigg|\bigg(\frac{P_\varepsilon}{P_\infty}\bigg)^{\frac{1}{p}}\hat{\phi}_{n_k}\bigg|^{p^\prime}dx-\int_{\R^N}w_{n_k}\mathbf{K}_{\varepsilon}w_{n_k}dx\\
\leq&\bigg(\frac{Q_\infty+\eta}{Q_\infty}\bigg)^{\frac{q^\prime}{q}}\int_{\{|x|\geq 2r\}}|\psi_{n_k}|^{q^\prime}dx+\bigg(\frac{P_\infty+\eta}{P_\infty}\bigg)^{\frac{p^\prime}{p}}\int_{\{|x|\geq 2r\}}|\phi_{n_k}|^{p^\prime}dx-\int_{\R^N}w_{n_k}\mathbf{K}_{\varepsilon}w_{n_k}dx\\
\leq&\bigg(1+\bigg({\frac{\eta}{Q_\infty}}\bigg)^{\frac{q^\prime}{q}}\bigg)\int_{\mathbb{R}^N}|\psi_{n_k}|^{q^\prime}dx+\bigg(1+\bigg({\frac{\eta}{P_\infty}}\bigg)^{\frac{p^\prime}{p}}\bigg)\int_{\mathbb{R}^N}|\phi_{n_k}|^{p^\prime}dx-\int_{\R^N}z_{n_k}\mathbf{K}_{\varepsilon}z_{n_k}dx\\
=&\bigg({\frac{\eta}{Q_\infty}}\bigg)^{\frac{q^\prime}{q}}\int_{\mathbb{R}^N}|\psi_{n_k}|^{q^\prime}dx+\bigg({\frac{\eta}{P_\infty}}\bigg)^{\frac{p^\prime}{p}}\int_{\mathbb{R}^N}|\phi_{n_k}|^{p^\prime}dx\\
\leq&\bigg(\bigg({\frac{\eta}{Q_\infty}}\bigg)^{\frac{q^\prime}{q}}+\bigg({\frac{\eta}{P_\infty}}\bigg)^{\frac{p^\prime}{p}}\bigg)C_1\ \rightarrow 0 \quad\text{ as } \eta\rightarrow0^+,
\end{split}
\end{equation*} 
where also used \eqref{eq3.14} and the inequality $(a+b)^\beta\leq a^\beta+b^\beta$ for all $a,b>0$ and $0<\beta<1$.
In addition, $t_{k,\infty}\tilde{w}_k\in \mathcal{N}_\infty $ implies that $h(t_{k,\infty})=0$. If $t_{k,\infty}>1$ and we have $0=h(t_{k,\infty})<h(1)\leq(({\frac{\eta}{Q_\infty}})^{\frac{q^\prime}{q}}+({\frac{\eta}{P_\infty}})^{\frac{p^\prime}{p}})C_1\rightarrow0$ as $\eta\rightarrow0^+$, this is impossible.
Consequently, the above estimates imply that for all $k \geq k_0$,
\begin{equation*}
\begin{split}
c_\infty\leq& J_{\infty}\left(t_{k,\infty}\tilde{w}_k\right)\\
\leq&\bigg(1+\frac{\eta}{\alpha}\bigg)^{\max\{\frac{q^\prime}{q},\frac{p^\prime}{p}\}}\bigg[\left(\frac{1}{q^{\prime}}-\frac{1}{2}\right)\int_{\mathbb{R}^N}|\psi_{n_k}|^{q^{\prime}}dx + \left(\frac{1}{p^{\prime}}-\frac{1}{2}\right)\int_{\mathbb{R}^N}|\phi_{n_k}|^{q^{\prime}}dx \bigg]\\
=&\bigg(1+\frac{\eta}{\alpha}\bigg)^{\max\{\frac{q^\prime}{q},\frac{p^\prime}{p}\}}J_\varepsilon(z_{n_k}).
\end{split}
\end{equation*}
Letting $k\rightarrow\infty$, we have
$$c_\infty\leq\bigg(1+\frac{\eta}{\alpha}\bigg)^{\max\{\frac{q^\prime}{q},\frac{p^\prime}{p}\}}d,$$
and letting $\eta\rightarrow0$ we obtain that
$$c_\infty\leq d,$$
which contradicts the assumption $c_{\varepsilon}\leq d< c_{\infty}$. Therefore, \eqref{eq3.10} holds and $z_n\rightarrow z$ in $X$ as $n\rightarrow\infty$.
\hfill
\end{proof}

In this section, we conclude by establishing a representation Lemma for Palais-Smale sequences of the dual energy functional in the case where both $P$ and $Q$ are constants. This Lemma will help us obtain the concentration of dual ground state solutions. 
\begin{Lem}\label{Lem3.3}
Assume $P$, $Q$ are positive constants on $\R^N$, that is $P(t)>0$ and $Q(s)>0$. If $\{z_n\}\subset X$ is a $(PS)_d$-sequence for $J_{s,t}$ at some level $d>0$, then there exist $m\in \mathbb{N}$ nontrivial critical points $w^{(1)}, \ldots, w^{(m)}$ of $J_{s,t}$ and sequences $\{x^{(j)}_{n}:1\leq j\leq m \}\subset \R^N$ such that (up to a subsequence)
\begin{itemize}
\item[$(1)$]  $\|z_n-\sum\limits_{j=1}^m w^{(j)}(\cdot-x_n^{(j)})\| \rightarrow 0\  \text { as } n \rightarrow \infty$;
\item[$(2)$]  $\left|x_n^{(i)}-x_n^{(j)}\right| \rightarrow \infty \  \text { as } n \rightarrow \infty \text {, if  }\ i \neq j $;
\item[$(3)$] $\sum\limits_{j=1}^m J_{s,t}\left(w^{(j)}\right)=d$.
\end{itemize}
\end{Lem}
\begin{proof}
Let $\{z_n\}=\{(\psi_n,\phi_n)^T\}\subset X$ be a $(PS)_d$-sequence for $J_{s,t}$ at level $d>0$, we know that $\{z_n\}$ is bounded and there holds 
\begin{equation*}
\begin{split}
&\bigg(1-\frac{1}{p}-\frac{1}{q}\bigg)\int_{\R^N}\psi_n\mathbf{K}_{q,p}^{Q(s),P(t)}\phi_ndx=\bigg(1-\frac{1}{p}-\frac{1}{q}\bigg)\int_{\R^N}\phi_n\mathbf{K}_{p,q}^{P(t),Q(s)}\psi_ndx\\
&=J_{s,t}(z_n)-\frac{1}{q^\prime}\langle J_{s,t}^\prime(z_n),(\psi_n,0)^T \rangle-\frac{1}{p^\prime}\langle J_{s,t}^\prime(z_n),(0,\phi_n)^T \rangle\\
&\longrightarrow d>0\quad \text{as}\ n\rightarrow\infty.
\end{split}
\end{equation*}
Thus, by the nonvanishing theorem\cite[Theorem 2.7]{Ding-Wang2023JDE}, there exists $R$, $\zeta>0$ and a sequence $\{x^{(1)}_n\}\subset \R^N$ such that, up to sequence,
\begin{equation}\label{eq3.21}
\text{either}\quad \int_{B_R(x^{(1)}_n)}\left|\psi_n\right|^{q^{\prime}} dx \geq\zeta \quad \text { or } \quad \int_{B_R (x^{(1)}_n)}\left|\phi_n\right|^{p^{\prime}} dx\geq \zeta \quad \text { for all } n .
\end{equation}
Considering the following sequence:
$$\psi^{(1)}_n(x)=\psi_n(x+x^{(1)}_n),\quad \phi^{(1)}_n(x)=\phi_n(x+x^{(1)}_n),\quad z_n^{(1)}=(\psi^{(1)}_n,\phi^{(1)}_n)^T.$$
Note that $P_t,Q_s$ are positive constants, by the translation invariance, we find that $z_n^{(1)}$ is also a $(PS)_d$-sequence for $J_{s,t}$. Up to a subsequence, there exists $w^{(1)}=(\psi^{(1)},\phi^{(1)})^T\in X$ such that $z_n^{(1)}\rightharpoonup w^{(1)}$ in $X$, $\psi^{(1)}_n\rightarrow \psi^{(1)}$ in $L_{loc}^{q^\prime}(\R^N)$ and $\phi^{(1)}_n\rightarrow \phi^{(1)}$ in $L_{loc}^{p^\prime}(\R^N)$. Furthermore, $J_{s,t}(w^{(1)})\leq \lim\limits_{n\rightarrow\infty}J_{s,t}(z_n^{(1)})=d$ and $w^{(1)}$ is a nontrivial critical point of $J_{s,t}$. Next, we distinguish two cases:
\par
\textbf{Case 1:} If $J_{s,t}(w^{(1)})=d$, using the fact that $w^{(1)}$ is a nontrivial critical point of $J_{s,t}$, we have 
\begin{equation*}
\begin{split}
\bigg(1-\frac{1}{q}-\frac{1}{p}\bigg)\|\psi^{(1)}\|_{q^\prime}^{q^\prime}&=J_{s,t}(w^{(1)})-\frac{1}{p^\prime}\langle J_{s,t}^\prime(w^{(1)}),(0,\phi^{(1)})^T\rangle=d\\
&=\lim\limits_{n\rightarrow\infty}[J_{s,t}(z_n^{(1)})-\frac{1}{p^\prime}\langle J_{s,t}^\prime(z_n^{(1)}),(0,\phi^{(1)}_n)^T\rangle]\\
&=\bigg(1-\frac{1}{q}-\frac{1}{p}\bigg)\lim\limits_{n\rightarrow\infty}\|\psi^{(1)}_n\|_{q^\prime}^{q^\prime}.
\end{split}
\end{equation*}
Similarly, $\|\phi^{(1)}\|_{p^\prime}^{p^\prime}=\lim\limits_{n\rightarrow\infty}\|\phi^{(1)}_n\|_{p^\prime}^{p^\prime}$, and hence $z_n^{(1)}\rightarrow w^{(1)}$ in $X$ as $n\rightarrow\infty$, we finish the proof of this Lemma.
\par
\textbf{Case 2:} If $J_{s,t}(w^{(1)})<d$, we set $z_n^{(2)}=z_n^{(1)}-w^{(1)}=(\psi^{(2)}_n,\phi^{(2)}_n)^T$, where $\psi^{(2)}_n:=\psi^{(1)}_n-\psi^{(1)}$, $\phi^{(2)}_n:=\phi^{(1)}_n-\phi^{(1)}$. We now prove that $z_n^{(2)}$ is a $(PS)$-sequence for $J_{s,t}$ at level $d-J_{s,t}(w^{(1)})>0$. Indeed, by the Br\'{e}zis-Lieb Lemma, we know that 
\begin{equation}\label{eq3.22}
\begin{split}
\int_{\R^N}|\psi^{(2)}_n|^{q^\prime}dx&=\int_{\R^N}|\psi^{(1)}_n|^{q^\prime}dx-\int_{\R^N}|\psi^{(1)}|^{q^\prime}dx+o_n(1),\\
\int_{\R^N}|\phi^{(2)}_n|^{p^\prime}dx&=\int_{\R^N}|\phi^{(1)}_n|^{p^\prime}dx-\int_{\R^N}|\phi^{(1)}|^{p^\prime}dx+o_n(1).
\end{split}
\end{equation}
Using the fact that $\psi^{(1)}_n\rightharpoonup \psi^{(1)}$ in $L^{q^\prime}(\R^N)$ and $\phi^{(1)}_n\rightharpoonup \phi^{(1)}$ in $L^{p^\prime}(\R^N)$, we have 
\begin{equation}\label{eq3.23}
\int_{\R^N}\psi^{(1)}_n\mathbf{K}_{q,p}^{Q(s),P(t)}\phi^{(1)}_ndx-\int_{\R^N}\psi^{(1)}\mathbf{K}_{q,p}^{Q(s),P(t)}\phi^{(1)}dx=\int_{\R^N}\psi^{(2)}_n\mathbf{K}_{q,p}^{Q(s),P(t)}\phi^{(2)}_ndx+o_n(1).
\end{equation}
\eqref{eq3.22} and \eqref{eq3.23} imply that $J_{s,t}(z_n^{(2)})=J_{s,t}(z_n^{(1)})-J_{s,t}(w^{(1)})+o_n(1)=d-J_{s,t}(w^{(1)})+o_n(1)$ as $n\rightarrow\infty$. Note that $\psi^{(1)}_n-\psi^{(2)}_n-\psi^{(1)}=0$, $\phi^{(1)}_n-\phi^{(2)}_n-\phi^{(1)}=0$ and similar proofs in \cite[Lemma 2.3]{Evequoz2017AMPA}, we have $J_{s,t}^\prime(z_n^{(2)})=J_{s,t}^\prime(z_n^{(1)})-J_{s,t}^\prime(w^{(1)})+o_n(1)=o_n(1)$ as $n\rightarrow\infty$. Therefore, $\{z^{(2)}_n\}$ is a $(PS)_{d-J_{s,t}(w^{(1)})}$-sequence for $J_{s,t}$.
\par
Due to the fact that $d-J_{s,t}(w^{(1)})>0$, continuing to use the nonvanishing theorem, we can obtain two positive numbers $R_1,\zeta_1$ and a sequence $\{y_n\}\subset\R^N$ such that, up to sequence, 
\begin{equation*}
\text{either}\quad \int_{B_{R_1}(y_n)}\left|\psi^{(2)}_n\right|^{q^{\prime}} dx \geqslant\zeta_1 \quad \text { or } \quad \int_{B_{R_1} (y_n)}\left|\phi^{(2)}_n\right|^{p^{\prime}} dx\geqslant \zeta_1 \quad \text { for all } n .
\end{equation*}
Obviously, there exists a critical point $w^{(2)}=(\psi^{(2)},\phi^{(2)})\in X$ of $J_{s,t}$ that satisfies $z_n^{(2)}\rightharpoonup w^{(2)} $ in $X$, $\psi^{(1)}_n\rightarrow \psi^{(1)}$ in $L_{loc}^{q^\prime}(\R^N)$ and $\phi^{(1)}_n\rightarrow \phi^{(1)}$ in $L_{loc}^{p^\prime}(\R^N)$ as $n\rightarrow\infty$. Setting $x_n^{(2)}=x_n^{(1)}+y_n$, due to $z_n^{(2)}\rightharpoonup0$, we have $|y_n|=|x_n^{(2)}-x_n^{(1)}|\rightarrow\infty$ as $n\rightarrow\infty$. It is easy to see that 
$$z_n(\cdot)-(w^{(1)}(\cdot-x_n^{(1)})+w^{(2)}(\cdot-x_n^{(2)}))=z_n^{(2)}(\cdot+y_n-x_n^{(2)})-w^{(2)}(\cdot-x_n^{2})\rightharpoonup0 $$
in $X$ as $n\rightarrow\infty$. In addition, $J_{s,t}(w^{(2)})\leq\liminf\limits_{n\rightarrow\infty}J_{s,t}(z_n^{(2)})=d-J_{s,t}(w^{(1)})$. If $J_{s,t}(w^{(2)})=d-J_{s,t}(w^{(1)})$, similar to Case 1, we can obtain that $z_n^{(2)}(\cdot+y_n)\rightarrow w^{(2)}$ in $X$. Otherwise, we can continue iterating as in Case 2 and the iteration has to stop after finitely many steps, since for every nontrivial critical point $w$ of $J_{s,t}$ we have $J_{s,t}(w)\geq c_{s,t}>0$.
\end{proof}

\section{The proof of Theorem 1.1 and Theorem 1.2}
In this last section, we will prove Theorem \ref{Thm1.1} and Theorem \ref{Thm1.2}. We first establish the concentration of the sequences of ground states for the dual variational functional $J_\varepsilon$ as $\varepsilon=k^{-1}\rightarrow0^+$.
\begin{Lem}\label{Lem4.1}
Let $\{\varepsilon_n\}\subset(0,\infty)$ with $\varepsilon_n\rightarrow0$ as $n\rightarrow\infty$, and for each $n\in \mathbb{N}$ some $z_n=(\psi_n,\phi_n)\in\mathcal{N}_{\varepsilon_n}$, $\lim\limits_{n\rightarrow\infty}J_{\varepsilon_n}(z_n)=c_M$. Then, there exists a critical point $w_0$ of $J_0$ at level $c_M$ and a sequence $\{y_n\}\subset\R^N$ such that, up to sequence,
\begin{equation}\label{eq4.1}
\varepsilon_n y_n\rightarrow x_0\in M_P\cap M_Q \quad \text {and}\quad\|z_n(\cdot+y_n)-w_0\|\rightarrow0 
\quad\text{as}\ n\rightarrow\infty.
\end{equation}
\end{Lem}

\begin{proof}
For each $n\in \mathbb{N}$, consider $z_{0,n}=(\psi_{0,n},\phi_{0,n})=((\frac{Q_{\varepsilon_n}}{\bar{Q}})^{\frac{1}{q}}\psi_n,(\frac{P_{\varepsilon_n}}{\bar{P}})^{\frac{1}{p}}\phi_n)$. Similar to the proof of  Lemma \ref{Lem2.6} and Lemma \ref{Lem3.2}, we know that $z_{0,n}\in U_{0}^+$ and there exist $0<t_{0,n}\leq 1 $ such that $t_{0,n}z_{0,n}\in \mathcal{N}_0$. We first claim that $\{t_{0,n}z_{0,n}\}$ is a minimizing sequence for $J_0$ on $\mathcal{N}_0$, i.e. $J_0(t_{0,n}z_{0,n})\rightarrow c_M=\inf\limits_{z\in \mathcal{N}_0}J_0(z)$ as $n\rightarrow\infty$. In fact, using the fact that $z_n\in\mathcal{N}_{\varepsilon_n}$ and $t_{0,n}z_{0,n}\in \mathcal{N}_0$, we have 
\begin{equation*}
\begin{split}
c_{ M}&\leq J_0(t_{0,n}z_{0,n})=\bigg(\frac{1}{q^\prime}-\frac{1}{2}\bigg)t_{0,n}^{q^\prime}\|\psi_{0,n}\|_{q^\prime}^{q^\prime}+\bigg(\frac{1}{p^\prime}-\frac{1}{2}\bigg)t_{0,n}^{p^\prime}\|\phi_{0,n}\|_{p^\prime}^{p^\prime}\\
&\leq\max\{t^{q^\prime}_{0,n},t^{p^\prime}_{0,n}\}\bigg[\bigg(\frac{1}{q^\prime}-\frac{1}{2}\bigg)\|\psi_n\|_{q^\prime}^{q^\prime}+\bigg(\frac{1}{p^\prime}-\frac{1}{2}\bigg)\|\phi_n\|_{p^\prime}^{p^\prime}\bigg]\\
&\leq J_{\varepsilon_n}(z_n)\rightarrow c_M \quad \text{as}\ n\rightarrow\infty.
\end{split}
\end{equation*}
Hence, $\{t_{0,n}z_{0,n}\}$ is a minimizing sequence for $J_0$ on $\mathcal{N}_0$, combined with the Ekeland's variational principle, there is a $(PS)_{c_M}$-sequence $\{w_n\}=\{(\bar{\psi}_n,\bar{\phi}_n)\}\subset X$ such that $\|z_{0,n}-w_n\|\rightarrow0$ as $n\rightarrow\infty$. From the Case 1 of Lemma \ref{Lem3.3}, we know that there exists a critical point $w_0=(\psi_0,\phi_0)\in  X$ for $J_0$ at the level $c_M$ and a sequence $\{y_n\}\subset\R^N$ such that $\|w_n(\cdot+y_n)-w_0\|\rightarrow0$ as $n\rightarrow\infty$. Hence, $z_{0,n}(\cdot+y_n)\rightarrow w_0$ in $X$ as $n\rightarrow\infty$, i.e., 
\begin{equation}\label{eq4.2}
\psi_{0,n}(\cdot+y_n)\rightarrow \psi_0\  \text{in}\  L^{q^\prime}(\R^N) \quad  \text{and}\quad  \phi_{0,n}(\cdot+y_n)\rightarrow \phi_0 \  \text{in}\  L^{p^\prime}(\R^N).
\end{equation} 
\par
Next, we claim that the sequence $\{\varepsilon_n y_n\}$ is bounded. If not, there exists some subsequence still denoted $\{\varepsilon_n y_n\}$ that satisfies  $\lim\limits_{n\rightarrow\infty}|\varepsilon_n y_n|=+\infty$. By $(PQ_1)$, we know that $Q_\infty\neq0$ and $P_\infty\neq0$. Using $(PQ_2)$, Fatou's Lemma and \eqref{eq4.2}, we get
\begin{equation*}
\begin{split}
c_M=&\lim\limits_{n\rightarrow\infty} J_{\varepsilon_n}(z_n)=\lim\limits_{n\rightarrow\infty}\bigg[\bigg(\frac{1}{q^\prime}-\frac{1}{2}\bigg)\|\psi_n\|_{q^\prime}^{q^\prime}+\bigg(\frac{1}{p^\prime}-\frac{1}{2}\bigg)\|\phi_n\|_{p^\prime}^{p^\prime}\bigg]\\
=&\lim\limits_{n\rightarrow\infty}\bigg[\bigg(\frac{1}{q^\prime}-\frac{1}{2}\bigg)\|\psi_n(x+y_n)\|_{q^\prime}^{q^\prime}+\bigg(\frac{1}{p^\prime}-\frac{1}{2}\bigg)\|\phi_n(x+y_n)\|_{p^\prime}^{p^\prime}\bigg]\\
\geq&\bigg(\frac{1}{q^\prime}-\frac{1}{2}\bigg)\liminf\limits_{n\rightarrow\infty}\int_{\R^N}\bigg(\frac{\bar{Q}}{Q(\varepsilon_n x+\varepsilon_ny_n)}\bigg)^{\frac{q^\prime}{q}}|\psi_{0,n}(x+y_n)|^{q^\prime}dx\\
&+\bigg(\frac{1}{p^\prime}-\frac{1}{2}\bigg)\liminf\limits_{n\rightarrow\infty}\int_{\R^N}\bigg(\frac{\bar{P}}{P(\varepsilon_n x+\varepsilon_ny_n)}\bigg)^{\frac{p^\prime}{p}}|\psi_{0,n}(x+y_n)|^{p^\prime}dx\\
\geq& \bigg(\frac{1}{q^\prime}-\frac{1}{2}\bigg)\bigg(\frac{\bar{Q}}{Q_\infty}\bigg)^{\frac{q^\prime}{q}}\|\psi_0\|_{q^\prime}^{q^\prime}+\bigg(\frac{1}{p^\prime}-\frac{1}{2}\bigg)\bigg(\frac{\bar{P}}{P_\infty}\bigg)^{\frac{p^\prime}{p}}\|\phi_0\|_{p^\prime}^{p^\prime}\\
>&\bigg(\frac{1}{q^\prime}-\frac{1}{2}\bigg)\|\psi_0\|_{q^\prime}^{q^\prime}+ \bigg(\frac{1}{p^\prime}-\frac{1}{2}\bigg)\|\phi_0\|_{p^\prime}^{p^\prime}=J_0(w_0)=c_M,
\end{split}
\end{equation*}
which is a contradiction. Thus $\{\varepsilon_n y_n\}$ is bounded, up to a subsequence, we assume that $\varepsilon_n y_n\rightarrow x_0\in\R^N$. 
\par
Finally, we conclude the proof of this Lemma by proving $x_0\in  M_P\cap M_Q$. Note that $J_0(w_0)=c_M$ and $J_0^\prime(w_0)=0$, by the dominated convergence Theorem we have 
\begin{equation*}
\begin{split}
&\bigg(\frac{1}{q^\prime}-\frac{1}{2}\bigg)\|\psi_0\|_{q^\prime}^{q^\prime}+\bigg(\frac{1}{p^\prime}-\frac{1}{2}\bigg)\|\phi_0\|_{q^\prime}^{q^\prime}=c_M=\lim\limits_{n\rightarrow\infty}J_{\varepsilon_n}(z_n)\\
=&\lim\limits_{n\rightarrow\infty}\bigg[\bigg(\frac{1}{q^\prime}-\frac{1}{2}\bigg)\int_{\R^N}\bigg(\frac{\bar{Q}}{Q(\varepsilon_n x+\varepsilon_ny_n)}\bigg)^{\frac{q^\prime}{q}}|\psi_{0,n}(x+y_n)|^{q^\prime} dx\\
&+\bigg(\frac{1}{p^\prime}-\frac{1}{2}\bigg)\int_{\R^N}\bigg(\frac{\bar{P}}{P(\varepsilon_n x+\varepsilon_ny_n)}\bigg)^{\frac{p^\prime}{p}}|\phi_{0,n}(x+y_n)|^{p^\prime} dx\bigg]\\
=&\bigg(\frac{1}{q^\prime}-\frac{1}{2}\bigg)\bigg(\frac{\bar{Q}}{Q(x_0)}\bigg)^{q^\prime-1}\|\psi_0\|_{q^\prime}^{q^\prime}+\bigg(\frac{1}{p^\prime}-\frac{1}{2}\bigg)\bigg(\frac{\bar{P}}{P(x_0)}\bigg)^{p^\prime-1}\|\phi_0\|_{q^\prime}^{q^\prime}.
\end{split}
\end{equation*}
Hence, $\bar{Q}=Q(x_0)$, $\bar{P}=P(x_0)$.
Consequently, using again the dominated convergence Theorem and \eqref{eq4.2}, we know 
\begin{equation*}
\begin{split}
&z_n(x+y_n)=\bigg(\bigg(\frac{\bar{Q}}{Q(\varepsilon_n x+\varepsilon_n y_n)}\bigg)^{\frac{1}{q}}\psi_{0,n}(x+y_n),\bigg(\frac{\bar{P}}{P(\varepsilon_n x+\varepsilon_n y_n)}\bigg)^{\frac{1}{p}}\phi_{0,n}(x+y_n)\bigg)\\
&\longrightarrow \bigg(\bigg(\frac{\bar{Q}}{Q(x_0)}\bigg)^{\frac{1}{q}}\psi_0, \bigg(\frac{\bar{P}}{P(x_0)}\bigg)^{\frac{1}{p}}\phi_0\bigg)=(\psi_0,\phi_0)=w_0\ \text{in}\ X \ \text{as}\ n\rightarrow\infty.
\end{split}
\end{equation*}
\hfill
\end{proof}
\par
\textbf{The proof of Theorem 1.1:} By $(PQ_2)$, we have $c_M<c_\infty$. Thus, it follows from Lemma \ref{Lem2.6} that there exists an $\varepsilon_0>0$ such that $c_\varepsilon<c_\infty$ for any $\varepsilon\in(0,\varepsilon_0)$. Ekeland's variational principle implies the existence of a $(PS)_{c_\varepsilon}$-sequence $\{z_n\}$ for $J_\varepsilon$ on $\mathcal{N}_\varepsilon$, and Lemma \ref{Lem3.2} implies that $\{z_n\}$ has a subsequence converging to $z_0\in X$ such that $J_\varepsilon(z_0)=c_\varepsilon$ and $J^\prime_\varepsilon(z_0)=0$.
\par
In order to show the concentration, setting $\varepsilon_n=k_n^{-1}$, for each $n$ and $x\in\R^N$, the dual ground state satisfies 
\begin{equation}\label{eq dual=}
w_n(x)=(u_n(x),v_n(x))=(k_n^{-\beta_1}\mathbf{R}(P_{\varepsilon_n}^{\frac{1}{p}}\phi_n)(k_n x),k_n^{-\beta_2}\mathbf{R}(Q_{\varepsilon_n}^{\frac{1}{q}}\psi_n)(k_n x)),\end{equation}
where $z_n=(\psi_n,\phi_n)\in X$ satisfies $J_{\varepsilon_n}(z_n)=c_{\varepsilon_n}$ and $J^\prime(z_n)=0$. It follows from Lemma \ref{Lem2.6} that $\lim\limits_{n\rightarrow\infty}J_{\varepsilon_n}(z_n)=\lim\limits_{n\rightarrow\infty} c_{\varepsilon_n}=c_M$, and Lemma \ref{Lem4.1} shows that there exists a pair $(x_0,x_0)\in M$ and a sequence $\{y_n\}\subset\R^N$ such that $x_n:=\varepsilon_n y_n\rightarrow x_0$ and $z_n(\cdot+y_n)\rightarrow w_0=(\psi_0,\phi_0)$ in $X$ as $n\rightarrow\infty$, where $w_0$ is a critical point for $J_0$ at the level $c_M$. Note that 
\begin{equation*}
\begin{split}
&(k^{\beta_1}u_n(k_n^{-1}x+x_n),k^{\beta_2}v_n(k_n^{-1}x+x_n))\\
&=\bigg(\mathbf{R}\bigg(P^{\frac{1}{p}}_{\varepsilon_n}\phi_n\bigg)(x+k_nx_n),\mathbf{R}\bigg(Q^{\frac{1}{q}}_{\varepsilon_n}\psi_n\bigg)(x+k_nx_n)\bigg)\\
&=\bigg(\mathbf{R}\bigg(P^{\frac{1}{p}}_{\varepsilon_n}(\cdot+y_n)\phi_n(\cdot+y_n)\bigg)(x),\mathbf{R}\bigg(Q^{\frac{1}{q}}_{\varepsilon_n}(\cdot+y_n)\psi_n(\cdot+y_n)\bigg)(x)\bigg)\\
&\rightarrow\bigg(\mathbf{R}\bigg(\bar{P}^{\frac{1}{p}}\phi_0\bigg),\mathbf{R}\bigg(\bar{Q}^{\frac{1}{q}}\psi_0\bigg)\bigg),
\end{split}
\end{equation*}
where we used the fact that the operator $\mathbf{R}$ and the functions $P,Q$ are continuous. In addition, setting
$$z_0=(u_0,v_0)=(\mathbf{R}(\bar{P}^{\frac{1}{p}}\phi_0),\mathbf{R}(\bar{Q}^{\frac{1}{q}}\psi_0)),$$
by $J_0(w_0)=c_M$ and $J^\prime_0(w_0)=0$ we know that $z_0$ is a dual ground state solution of the limit problem \eqref{eq1.5}. \qed
\par
To prove the multiplicity of dual ground state solutions, let $h(\varepsilon)$ be a positive function tending to $0$ as $\varepsilon\rightarrow0^+$ and $h(\varepsilon)>c_\varepsilon-c_M$ for all $\varepsilon>0$, we can define 
$$\Sigma_\varepsilon:=\{z\in \mathcal{N}_\varepsilon: J_\varepsilon(z)\leq c_M+h(\varepsilon)\}\subset\mathcal{N}_\varepsilon,$$
which is not empty for $\varepsilon>0$ small. Next, we will construct a mapping $\varphi_{\varepsilon,x_0}$ from $M$ to $\Sigma_\varepsilon$ and a mapping $\beta_\varepsilon$ from $\Sigma_\varepsilon$ to $M_\delta$ such that the composition $\beta_\varepsilon\circ\varphi_{\varepsilon,x_0}$ is homotopic to the inclusion $j:M\rightarrow M_\delta$. Let $\left(\psi_{\varepsilon, x_0}, \phi_{\varepsilon,x_0}\right)$ be as in Lemma \ref{Lem2.4}, we denote the mapping $\varphi_{\varepsilon,x_0}(x):M\rightarrow\Sigma_\varepsilon$ by 
$$\varphi_{\varepsilon,x_0}(x)=\left(\psi_{\varepsilon, x_0}, \phi_{\varepsilon,x_0}\right).$$
Let $\rho>0$ such that $M_\delta\subset B_\rho\times B_\rho$ and define $\chi:\R^N\rightarrow\R^N$ be a function given by 
$$\chi(x)= \begin{cases}x & \text{if}\ |x|\leq\rho, \\ \frac{\rho x}{|x|} & \text{if}\ |x|>\rho. \end{cases}$$
We define $\beta_\varepsilon:\Sigma_\varepsilon\rightarrow \R^N\times\R^N$ by 
$$\beta_\varepsilon(z)=(\beta_{1,\varepsilon}(\psi),\beta_{2,\varepsilon}(\phi))=\bigg(\frac{\int_{\R^N}\chi(\varepsilon x)|\psi(x)|^{q^\prime}dx}{\|\psi\|_{q^\prime}^{q^\prime}},\frac{\int_{\R^N}\chi(\varepsilon x)|\phi(x)|^{p^\prime}dx}{\|\phi\|_{p^\prime}^{p^\prime}}\bigg),$$
which corresponds to a local center of mass of $(\psi,\phi)$.
\par 
Changing variables, by the continuity of $\chi$ we know that 
\begin{equation}\label{eq4.3}
\begin{split}
\beta_{1,\varepsilon}(\psi_{\varepsilon,x_0})&=\frac{\int_{\R^N}\chi(\varepsilon x)|\psi_{\varepsilon,x_0}|^{q^\prime}dx}{\int_{\R^N}|\psi_{\varepsilon,x_0}|^{q^\prime}dx}=\frac{t^{q^\prime}_{\varepsilon,x_0}\int_{\R^N}(\frac{Q(x_0)}{Q(\varepsilon z+x_0)})^{\frac{q^\prime}{q}}\chi(\varepsilon z+x_0)|\eta(|\varepsilon z|)U(z)|^{q^\prime}dz}{t^{q^\prime}_{\varepsilon,x_0}\int_{\R^N}(\frac{Q(x_0)}{Q(\varepsilon z+x_0)})^{\frac{q^\prime}{q}}|\eta(|\varepsilon z|)U(z)|^{q^\prime}dz}\\
&\longrightarrow \chi(x_0)=x_0\quad \text{as}\ \varepsilon\rightarrow0^+.
\end{split}
\end{equation}
Similarly, we get $\beta_{2,\varepsilon}(\phi_{\varepsilon,x_0})=x_0+o(1)$. Therefore, for each pair of points $(x_0,x_0)\in M$, we have $\lim\limits_{\varepsilon\rightarrow0}\beta_\varepsilon(\varphi_{\varepsilon,x_0}(x))=(x_0,x_0)$. In addition, the limit holds uniformly for $(x_0,x_0)\in M_\delta$ due to the compactness of $M_\delta$.
\begin{Lem}\label{Lem4.2}
$$\lim\limits_{\varepsilon\rightarrow0^+}\sup\limits_{z\in \Sigma_\varepsilon}\inf\limits_{\xi\in M_{\frac{2}{\delta}}}|\beta_\varepsilon(z)-\xi|=0.$$
\end{Lem}
\begin{proof}
Let $\{\varepsilon_n\}\subset(0,\infty)$ be any sequence that satisfies $\lim\limits_{n\rightarrow\infty}\varepsilon_n=0$. For any $n$ there exists $z_n\in \Sigma_{\varepsilon_n}$ such that 
\begin{equation}\label{eq4.5}
\inf\limits_{\xi\in M_{\frac{\delta}{2}}}|\beta_{\varepsilon_n}(z_n)-\xi|\geq \sup\limits_{z\in \Sigma_{\varepsilon_n}}\inf\limits_{\xi\in M_{\frac{\delta}{2}}}|\beta_{\varepsilon_n}(z_n)-\xi|-\frac{1}{n}.
\end{equation}
Lemma \ref{Lem4.1} ensures that there exists a pair $(x_0,x_0)\in M$ and a sequence $\{y_n\}\subset\R^N$ with the property $\varepsilon_n y_n\rightarrow x_0$ and $z_n(\cdot+y_n)\rightarrow w_0$ in $X$ as $n\rightarrow\infty$, where $w_0$ is the critical point for $J_0$ at level $c_M$. Arguing as in the proof \eqref{eq4.3}, we can conclude that 
\begin{equation}\label{eq4.4}
\lim\limits_{n\rightarrow\infty}\beta_{\varepsilon_n}(z_n)=(\chi(x_0),\chi(x_0))=(x_0,x_0).
\end{equation}
Thus, from \eqref{eq4.5} and \eqref{eq4.4} we obtain that 
$$0\leq\lim\limits_{n\rightarrow\infty}\sup\limits_{z\in \Sigma_{\varepsilon_n}}\inf\limits_{\xi\in M_{\frac{\delta}{2}}}|\beta_{\varepsilon_n}(z)-\xi|\leq\inf\limits_{\xi\in M_{\frac{\delta}{2}}}|(x_0,x_0)-\xi|=0.$$
Based on the arbitrariness of $\{\varepsilon_n\}$, we have completed the proof. 
\hfill
\end{proof}
\par
To prove Theorem \ref{Thm1.2}, we will use a Lemma connecting the relative category and the multiplicity of critical points, for the proof, see\cite[Lemma 2.2]{Cingolani-Lazzo2000JDE}. 
\par
\textbf{The proof of Theorem 1.2:} Let $\delta>0$, the homotopy mapping is as follows:
$$M\stackrel{\varphi_{\varepsilon,x_0}}{\longrightarrow}\Sigma_\varepsilon\stackrel{\beta_{\varepsilon}}{\longrightarrow}M_\delta$$
Firstly, we prove that $\varphi_{\varepsilon,x_0}(M)\subset \Sigma_\varepsilon$. Indeed, by Lemma \ref{Lem2.5} we know $J_\varepsilon(\psi_{\varepsilon,x_0},\phi_{\varepsilon,x_0})=J_\varepsilon(\varphi_{\varepsilon,x_0})\leq c_M+h(\varepsilon)$ for $(x_0,x_0)\in M$, and the existence of function $h(\varepsilon)$ that satisfies the required properties can be obtained from Lemma \ref{Lem2.6}. Thus, $\varphi_{\varepsilon,x_0}(M)\subset \Sigma_\varepsilon$.
By Lemma \ref{Lem4.2}, we obtain 
\begin{equation}\label{eq4.6}
\sup\limits_{z\in \Sigma_\varepsilon}\inf\limits_{\xi\in M_{\frac{2}{\delta}}}|\beta_\varepsilon(z)-\xi|<\frac{\delta}{2},\quad \forall \varepsilon<\varepsilon_0
\end{equation}
for some $\varepsilon_0>0$. \eqref{eq4.6} yields that $\text{dist}(\beta_\varepsilon(z),M)\leq\delta$ for any $z\in \Sigma_\varepsilon$, thus $\beta_\varepsilon(\Sigma_\varepsilon)\subset M_\delta$ and the map $\beta_\varepsilon\circ \varphi_{\varepsilon,x_0}$ is homotopic to the inclusion $j:M\rightarrow M_\delta$. Using the Lemma 2.2 obtained in \cite{Cingolani-Lazzo2000JDE}, we know that $\text{cat}_{\Sigma_\varepsilon}(\Sigma_\varepsilon)\geq\text{cat}_{M_\delta}(M)$ for $0<\varepsilon<\varepsilon_0$. The existence of at least $\text{cat}_{M_\delta}(M)$ distinct critical points of $J_\varepsilon$ can be obtained from the Lusternik-Schnirelmann theory for $C^1$- manifolds from \cite{Ribarska-Tsachev-Krastanov1995}. In addition, the concentration of solutions can be obtained as the proof of Theorem \ref{Thm1.1}.\qed

\section*{Acknowledgments}
The authors would like to thank the referees for his/her careful reading of the original manuscript and for his/her valuable comments concerning its improvement.
\bibliographystyle{plain}
\bibliography{Reference}
\end{document}